\documentclass[12pt]{amsart}
\usepackage{setspace}
\usepackage{fullpage}
\usepackage[draft=false, pdfusetitle=true, pdfborder={0 0 0}]{hyperref}
\newcommand\bartoncrossref[2]{\hyperref[#2]{#1~\ref*{#2}}}
\newcommand\bartoneqref[2]{\hyperref[#2]{#1~\textup{(\ref*{#2})}}}
\usepackage{graphicx}
\newcommand{\mypic}[1]{\includegraphics[draft=false]{bartonwardpics-#1.mps}}

\numberwithin{equation}{section}
\numberwithin{figure}{section}

\newtheorem{theorem}{Theorem}[section]
\newtheorem{lemma}[theorem]{Lemma}

\theoremstyle{definition}
\newtheorem{definition}[theorem]{Definition}
\newtheorem{example}[theorem]{Example}

\theoremstyle{remark}
\newtheorem{remark}[theorem]{Remark}

\theoremstyle{remark}

\theoremstyle{plain}
\newtheorem{conjecture}[theorem]{Conjecture}

\newcommand\thmref[1]{\bartoncrossref{Theorem}{thm:#1}}
\newcommand\lemref[1]{\bartoncrossref{Lem\-ma}{lem:#1}}

\newcommand\dfnref[1]{\bartoncrossref{Definition}{dfn:#1}}
\newcommand\rmkref[1]{\bartoncrossref{Remark}{rmk:#1}}
\newcommand\conjref[1]{\bartoncrossref{Conjecture}{conj:#1}}
\newcommand\figref[1]{\bartoncrossref{Figure}{fig:#1}}
\newcommand\secref[1]{\bartoncrossref{Section}{sec:#1}}

\usepackage{amssymb} % For \subsetneq

\newcommand{\abs}[1]{\lvert#1\rvert}

\newcommand{\D}{\mathbb{D}}
\newcommand{\R}{\mathbb{R}}
\newcommand{\C}{\mathbb{C}}
\newcommand{\U}{\mathbb{U}}
\newcommand{\F}{\mathcal{F}}
\newcommand{\B}{\mathcal{B}}
\DeclareMathOperator{\dist}{dist}
\DeclareMathOperator{\diam}{diam}
\DeclareMathOperator{\re}{Re}
\DeclareMathOperator{\im}{Im}

\newcounter{tempctr}
\newenvironment{eqnenumerate}{%
	\setcounter{tempctr}{\value{equation}}
	
	\begin{enumerate}
	\usecounter{equation}
	\setcounter{equation}{\value{tempctr}}
}{%
	\setcounter{tempctr}{\value{equation}}
	\end{enumerate}
	\setcounter{equation}{\value{tempctr}}
}

\newenvironment{thmenumerate}{\begin{eqnenumerate}}{\end{eqnenumerate}}

\begin{document}

\title[{A new class of harmonic measure distribution functions}]{A new class of \\ harmonic measure distribution functions}

%    Information for first author
\author{Ariel Barton}
%\address{Department of Mathematics\\ Harvey Mudd College\\ Claremont, California 91711}
%\curraddr{Department of Mathematics, University of Minnesota, Minneapolis, Minnesota 55455}
\address{Department of Mathematics, University of Minnesota, Minneapolis, Minnesota 55455}
\email{abarton@math.umn.edu}
%    \thanks will become a 1st page footnote.
%\thanks{The first author was supported in part by NSF Grant \#000000.}

\author{Lesley A. Ward}
%\address{Department of Mathematics\\ Harvey Mudd College\\ Claremont, California 91711}
%\curraddr{School of Mathematics and Statistics, University of South Australia, Mawson Lakes Campus, Mawson Lakes SA 5095, Australia}
\address{School of Mathematics and Statistics, University of South Australia, Mawson Lakes Campus, Mawson Lakes SA 5095, Australia}
\email{Lesley.Ward@unisa.edu.au}

%    General info
% 30C85 Capacity and harmonic measure in the complex plane [See also 31A15]
% 30C20 Conformal mappings of special domains
% 31A15 Potentials and capacity, harmonic measure, extremal length [See also 30C85]
% 60J65 Brownian motion [See also 58J65]
%
% 60J70 Applications of Brownian motions and diffusion theory (population genetics, absorption problems, etc.) [See also 92Dxx]
% 60J70 is more applied than we are
%
% 2000 30E02

\date{}

%\dedicatory{This paper is dedicated to\ldots}


\begin{abstract}
Let $\Omega$ be a planar domain containing~0.  Let $h_\Omega(r)$ be the
harmonic measure at 0 in $\Omega$ of the part of the boundary of
$\Omega$ within distance $r$ of~0. The resulting function $h_\Omega$ is
called the harmonic measure distribution function of~$\Omega$. In this
paper we address the inverse problem by establishing several sets of
sufficient conditions on a function~$f$ for $f$ to arise as a harmonic
measure distribution function. In particular, earlier work of Snipes and
Ward shows that for each function $f$ that increases from zero to one,
there is a sequence of multiply connected domains $X_n$ such that
$h_{X_n}$ converges to $f$ pointwise almost everywhere. We show that if
$f$ satisfies our sufficient conditions, then $f = h_\Omega$, where
$\Omega$ is a subsequential limit of bounded simply connected domains
that approximate the domains~$X_n$. Further, the limit domain is unique
in a class of suitably symmetric domains. Thus $f =
h_\Omega$ for a unique symmetric bounded simply connected
domain~$\Omega$.
\end{abstract}

\maketitle

\section{Introduction}

Let $\Omega$ be a domain in the complex plane containing 0. We define the \emph{harmonic measure distribution function} $h_\Omega(r)$ by the formula
\[h_\Omega(r) := \omega(0,\partial\Omega\cap \overline{B(0,r)},\Omega)\]
where $\omega(0,E,\Omega)$ denotes the harmonic measure of the set $E$ from the basepoint 0 in~$\Omega$. For $r>0$, the number $h_\Omega(r)$ gives the probability that a Brownian particle released at 0 first exits $\Omega$ within a distance $r$ of~$0$. Two questions regarding such functions have been investigated in a number of papers \cite{WalW96, WalW01, BetS, SniW05, SniW08}.
First, what functions can be constructed as $h_\Omega$ for some~$\Omega$? Second, what can be determined about $\Omega$ from~$h_\Omega$?

These questions originally arose from Brannan and Hayman's 1989 paper \cite{BraH}, which described the current state of some problems in complex analysis and listed some new ones. Specifically, these questions arose from Problem 6.116, proposed by Stephenson. Problem 6.116 poses the questions listed above for the related function $w_\Omega(r)$, defined to be the harmonic measure of $\Omega\cap\partial B(0,r)$ in the connected component of $\Omega\cap B(0,r)$ containing~0.

The following informally stated theorem (restated as \thmref{mainunique} below) summarizes the results in this paper. This existence result gives sufficient conditions on a function $f$ for $f$ to arise as the harmonic measure distribution function of a bounded simply connected domain.

\begin{theorem}\label{thm:intro}
Let $f$ be a nondecreasing, right-continuous function that is zero on $(0,\mu)$ and $1$ on $[M,\infty)$, for some numbers $M>\mu>0$.
Suppose that the slopes of the secant lines to the graph of $f(r)$ for $\mu\leq r\leq M$ are bounded from below by a positive number.

If $f(r)$ is discontinuous at $r=\mu$, and if $M-\mu$ is small enough, then there exists a domain $\Omega$ that is bounded and simply connected, and is symmetric in the sense of \dfnref{symmetric}, such that $f=h_\Omega$. Furthermore, up to sets of harmonic capacity zero, this domain $\Omega$ is unique among bounded symmetric domains.
\end{theorem}
The precise condition on $M-\mu$ is given in the statement of \thmref{fjump}. The requirement that $f$ increase from $0$ to $1$ is necessary for $f$ to arise as the harmonic measure distribution function of \emph{any} bounded domain.
A set of less restrictive but more technical sufficient conditions is given in \thmref{main}.

Our proof of \thmref{intro} relies on an earlier existence result.
In \cite[Theorem~2]{SniW05}, Snipes and Ward proved that if $f$ is a right-continuous step function with finitely many jumps, and increases from~0 to~1, then $f$ arises as the harmonic measure distribution function of some multiply connected domain. They were able to describe this multiply connected domain fairly precisely. Much of the work in the current paper is done to pass from these multiply connected domains to simply connected domains.

We review some results concerning harmonic measure distribution functions of simply connected domains. In \cite{SniW08} and \cite{BetS}, it was shown that sequences of simply connected domains that converge in certain senses have convergent harmonic measure distribution functions. We will use these results; see \thmref{snipes:conv}.

In \cite{WalW96}, \cite{WalW01}, and \cite{BetS}, Walden, Ward, Betsakos and Solynin investigated the behavior of the harmonic measure distribution function $h_\Omega(r)$ of a simply connected domain for $r$ near $\dist(0,\partial\Omega)$. If $\mu=\dist(0,\partial\Omega)$ and $\Omega$ is simply connected, it is easy to see that $h_\Omega(r)=0$ for all $r<\mu$, and $h_\Omega(r)>0$ for all $r>\mu$. These papers established that for every number $\beta$ with $0<\beta<1$, there is some simply connected domain $\Omega$ such that $\lim_{r\to\mu^+}h_{\Omega}(r)/(r-\mu)^\beta$ exists and is positive. It was also shown that not all harmonic measure distribution functions of simply connected domains have this property (or even the weaker property $c(r-\mu)^\beta\leq h_{\Omega}(r)\leq C(r-\mu)^\beta$ for some numbers $C>c>0$).

These papers show that for some properties (asymptotic behavior of $f$ at $\mu$) there is a domain whose harmonic measure distribution function has those properties. They do not specify the whole function $f$; that is, they do not provide sufficient conditions for a function to arise as the harmonic measure distribution function of some domain. In this paper, we will provide three progressively more restrictive, but easier to check, conditions on~$f$, such that if $f$ meets any of those conditions then $f$ must be the harmonic measure distribution function of a simply connected domain. \thmref{intro} informally states the most transparent, and thus most restrictive, of these sufficient conditions.

\begin{example}\label{exa:jump}
The function
\begin{equation}
\label{eqn:jumpfunction}
f(r)=\begin{cases} 0, & 0< r\leq 1;\\
\displaystyle
\frac{1}{2}+\frac{1}{2}\frac{r-1}{0.0992}, & 1\leq r \leq 1.0992;\\
1, & 1.0992\leq r
\end{cases}
\end{equation}
graphed in \figref{jumpfunction} satisfies the sufficient conditions of \thmref{fjump} (see \rmkref{fjumpexample}); thus, we know that there exists a simply connected domain $\Omega$ such that $f=h_\Omega$. Before the current paper, it was not known whether this function, or indeed any function with a nonhorizontal linear segment, was the harmonic measure distribution function of any domain.
\end{example}

\begin{figure}
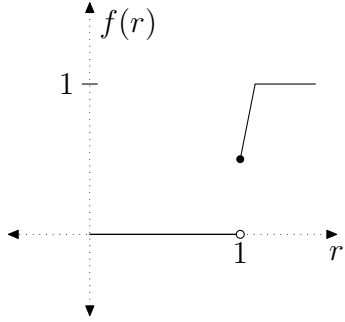

\mypic{8}
\caption{A function $f$, defined by \bartoneqref{equation}{eqn:jumpfunction}, that arises as the $h$-function of some simply connected domain.}
\label{fig:jumpfunction}
\end{figure}

The organization of this paper is as follows. In \secref{dfn}, we define our terms and give some background information about harmonic measure and harmonic measure distribution functions. In Sections~\ref{sec:convergence}, \ref{sec:sufficient} and~\ref{sec:functions}, we describe our sufficient conditions.

Specifically, in \secref{convergence}, we show that if $f=\lim_{n\to\infty} h_{\Omega_n}$, where $\{\Omega_n\}_{n=1}^\infty$ is a sequence of bounded, simply connected domains with uniformly locally connected complements, then $f=h_\Omega$ for some simply connected domain~$\Omega$.

In \secref{sufficient} we provide a candidate sequence~$\{\Omega_n\}_{n=1}^\infty$. Since monotonic functions can be approximated by step functions, by \cite{SniW05} (mentioned above) any function $f$ that increases from zero to one is the limit of a sequence of harmonic measure distribution functions of multiply connected domains. We provide a condition that lets us approximate these domains with well-behaved simply connected domains, yielding a sufficient condition on $f$ that may be checked by examining one sequence of domains rather than all sequences of simply connected domains.

In \secref{functions}, we show that functions that increase fast enough, in some senses that are easy to check, satisfy the sufficient condition of \secref{sufficient} and thus are harmonic measure distribution functions of simply connected domains. In particular, \bartoncrossref{Example}{exa:jump} and the existence part of \thmref{intro} are proven in \secref{functions}.

In \secref{unique} we consider uniqueness. From \cite[Remark~1]{SniW05}, a step function arises as the harmonic measure distribution function of many different domains. However, if the domain is assumed to be bounded and satisfy a symmetry condition, then the domain is unique. The symmetry condition can be generalized to arbitrary domains. We show that for bounded simply connected domains, the symmetry condition again implies uniqueness. We also show that the domains produced in Sections~\ref{sec:sufficient} and~\ref{sec:functions} satisfy this symmetry condition. Thus, if a function~$f$ satisfies the necessary conditions of Section~\ref{sec:sufficient} or~\ref{sec:functions}, then $f$ is the harmonic measure distribution function of a unique bounded simply connected symmetric domain.

Finally, in \secref{circle}, we prove some technical lemmas, involving the harmonic measure of circle domains and related domains, that we use in Sections~\ref{sec:sufficient} and~\ref{sec:functions}.

It should be emphasized that this paper provides only \emph{sufficient} conditions for a function $f$ to be realizable as the harmonic measure distribution function of a simply connected domain~$\Omega$. These conditions are not necessary.

Let $\F=\{h_\Omega:\Omega$ is simply connected$\}$ be the set of all harmonic measure distribution functions of simply connected domains, and let $\F_k$ be the set of all functions $f$ that satisfy the sufficient conditions of Section~$k$. Then $\F\supseteq \F_{\ref{sec:convergence}} \supseteq \F_{\ref{sec:sufficient}}\supseteq \F_{\ref{sec:functions}}$.

Some of these inclusions are proper.
It is easy to exhibit harmonic measure distribution functions that do not satisfy the conditions of \secref{functions}. In particular, no continuous function satisfies these conditions, and there are many well-behaved domains with continuous $h$-functions. For example, any disk containing but not centered at the origin has a continuous $h$-function.
By considering a constant sequence of domains, we can in fact exhibit functions $f=h_\Omega$ that satisfy the sufficient conditions of \secref{convergence} but not \secref{functions}. In other words, $\F_{\ref{sec:functions}}\subsetneq \F_{\ref{sec:convergence}}$. It is not known whether such functions satisfy the conditions of \secref{sufficient}.

Also, $\F_{\ref{sec:sufficient}}\subsetneq \F$.
The condition in \secref{sufficient} is sufficient for a function $f$ to arise as the harmonic measure distribution function of a bounded simply connected symmetric domain whose complement is locally connected. By uniqueness, if $f=h_\Omega$ for some bounded simply connected symmetric domain $\Omega$ whose complement is \emph{not} locally connected, then $f$ cannot satisfy the conditions of \secref{sufficient}.

An early version of this paper \cite{B04} appeared as the first author's undergraduate senior thesis, advised by the second author, while both authors were at Harvey Mudd College. Some of our later work on this paper was done while the first author was at the University of Chicago and at Purdue University.  The first author would like to thank Henry Krieger for acting as second reader of \cite{B04} and for useful advice regarding functional analysis.

\section{Definitions}\label{sec:dfn}

In this section, we will provide definitions for many terms and concepts used throughout this paper.

Let $\Omega$ be a connected domain, let $z$ be a point in~$\Omega$, and let $E$ be a measurable subset of $\partial\Omega$. Let $\omega(z,E,\Omega)$ denote the usual harmonic measure of $E$ in $\Omega$ with basepoint~$z$.
The harmonic measure can be calculated by solving the Dirichlet problem:
\begin{equation}\label{eqn:hmeasure}
\omega(z,E,\Omega) = u_E(z),\quad\text{where}\quad
\begin{cases}\Delta u_E=0 & \text{ in } \Omega, \\
u_E = 1_E & \text{ on } \partial\Omega.\end{cases}
\end{equation}
From \cite{Kak44}, $\omega(z,E,\Omega)$ is also the probability that a Brownian particle, released from $z$, first exits $\Omega$ through $E$.

Notice that if $\Phi:\Omega\mapsto \Phi(\Omega)$ is a conformal mapping that extends continuously to $\overline\Omega$, and if $F\subset\partial\Phi(\Omega)$ and $E=\Phi^{-1}(F)$, then
\[\omega(z,E,\Omega)=\omega(\Phi(z),F,\Phi(\Omega)).\]
In particular, if $\Phi$ is one-to-one on $\overline\Omega$ then $\omega(z,E,\Omega)=\omega(\Phi(z),\Phi(E),\Phi(\Omega))$ for every $E\subset\partial\Omega$.

Also, if $\Omega\subset\widetilde\Omega$, $E\subset\partial\Omega\cap\partial\widetilde\Omega$, and $z\in\Omega\cap\widetilde\Omega$, then
\[\omega(z,E,\Omega)\leq\omega(z,E,\widetilde\Omega).\]
This property is referred to as the \emph{monotonicity in the domain property} of harmonic measure.

Finally, if $\Omega\subset\widetilde\Omega$ and $E\subset\partial\widetilde\Omega$, then
\[\omega(z,E,\widetilde\Omega)\leq\omega(z,E\cap\partial\Omega,\Omega)
+\omega(z,\partial\Omega\setminus\partial\widetilde\Omega,\Omega).\]

We may define the harmonic measure distribution function, or $h$-function, of $\Omega$ by
\begin{equation}\label{eqn:h}
h_\Omega(r):=\omega(0,\partial\Omega\cap \overline{B(0,r)},\Omega).
\end{equation}

Recall that
\[w_\Omega(r) = \omega(0,\Omega\cap\partial B(0,r), \Omega^*)
=1-\omega(0,\partial\Omega\cap \overline{B(0,r)}, \Omega^*)\]
where $\Omega^*$ is the connected component of $\Omega\cap B(0,r)$ that contains~$0$.
Then by monotonicity in the domain, $h_\Omega(r)\geq 1-w_{\Omega}(r)$.

\begin{remark} \label{rmk:hnesc}
All harmonic measure distribution functions have a number of common properties. For any domain~$\Omega$, $h_\Omega$ is right-continuous, nondecreasing, and if $\mu=\dist(0,\partial\Omega)$, then $h_\Omega(r)=0$ for all $r<\mu$. Furthermore, if $\partial\Omega$ is bounded, then there is some $M$ such that $h_\Omega(r)=1$ for all $r\geq M$.

These observations give a set of necessary conditions for a function $f$ to arise as the harmonic measure distribution function of any domain.

Here is another necessary condition stemming from the relation between $h_\Omega$ and $w_\Omega$. This necessary condition applies only to the $h$-functions of simply connected domains. As noted in \cite{WalW96}, it follows from Beurling's solution to Milloux's problem \cite{Ahl73} that, if $\mu=\dist(0,\partial\Omega)$ and $\Omega$ is simply connected, then $h_\Omega(r)\geq 1-w_\Omega(r)\geq 1-w_{\Omega_{\mu,r}}(r)$, where $\Omega_{\mu,r}=B(0,r)\setminus [\mu,r)$. The domain $\Omega_{\mu,r}$ is simple enough that $w_{\Omega_{\mu,r}}$ may be computed directly via the Riemann map; we find that
\[h_{\Omega}(r)\geq 1-\frac{4}{\pi}\arctan\sqrt{\frac{\mu}{r}}.\]
For $r$ large we have the simpler inequality $h_\Omega(r)\geq 1-c/\sqrt{r}$ for some constant $c$; this inequality is known as Beurling's Lemma. As noted in \cite{WalW96}, this inequality shows that the condition $f=h_\Omega$ for some simply connected domain is more restrictive than the condition $f=h_\Omega$ for some arbitrary domain. (See \thmref{snipes:circle} for specific counterexamples.)
\end{remark}

We will not discuss necessary conditions further in this paper.

\begin{definition} \label{dfn:circle}
We say that $X\subset\C$ is a \emph{circle domain} if there exist numbers $0<r_0<r_1<\dots<r_{n}$ such that
\[X = B(0,r_{n})\setminus \bigcup_{j=0}^{n-1} A_j,\]
where $A_j$ is a connected closed proper subset of $\partial B(0,r_j)$ that is symmetric about the real axis and contains a positive real. We call the arcs $A_j$ the \emph{boundary arcs} of~$X$.

We may regard the boundary circle $\partial B(0,r_n)$ as the $n$th boundary arc~$A_n$; then $A_n$ is a full circle and $A_j$, $0\leq j<n$, is not.

We say that $\Omega\subset \C$ is a \emph{blocked circle domain} if $\Omega$ is simply connected, symmetric about the real axis,  $\Omega\subset X$ for some circle domain~$X$, and if $\partial\Omega\setminus\partial X$ consists of a number of \emph{gates}, that is, radial segments whose endpoints are at distances $r_j$ and $r_{j+1}$ from~0. See \figref{circledomain}.
\end{definition}

\begin{figure}
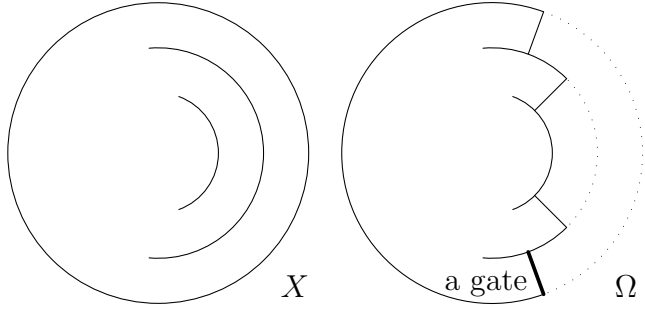

\mypic{0}
\caption{A circle domain $X$ and a blocked circle domain~$\Omega$.}
\label{fig:circledomain}
\end{figure}

Circle domains have historically been of interest because any step function satisfying the necessary conditions of \rmkref{hnesc} arises as the $h$-function of a circle domain. See \cite[Theorem~2]{SniW05}, quoted in this paper as \thmref{snipes:circle}. In the current paper, we will use sequences of blocked circle domains to construct domains with prescribed $h$-functions.

If $A$ is a connected arc of a circle centered at the origin, for notational convenience we refer to the angle at the origin, subtended by~$A$, as the \emph{arclength} of~$A$.

\begin{definition} \label{dfn:symmetric}
Throughout this paper, when we say that a domain $\Omega$ is \emph{symmetric}, we mean that $0\in\Omega$, that $\Omega$ is symmetric about the real axis, and that for every $r>0$, $\partial B(0,r)\setminus\Omega$ is either empty or a connected closed set that contains a point on the positive real axis.
\end{definition}
This notion of symmetry will be needed for the uniqueness results in \secref{unique}. We observe that a connected subset of $\partial B(0,r)$ is necessarily path-connected.

All circle domains and blocked circle domains are symmetric in this sense.
Let $\Omega$ be symmetric, and let $r$, $\theta$ be real numbers with $r>0$. We remark that if $re^{i\theta}\in\Omega$ then $-r\in\Omega$, and if $re^{i\theta}\notin\Omega$ then $r\notin\Omega$. If $\Omega$ is also bounded and connected, let $\mu$ and $M$ be the largest and smallest numbers, respectively, such that $B(0,\mu)\subset \Omega\subset B(0,M)$; then $\mu\notin\Omega$ and $(-M,\mu)\subset\Omega$. If in addition $\Omega$ or $\Omega^C$ is simply connected, then $\Omega\cap[\mu,M]=\emptyset$.

\begin{definition} \label{dfn:ULconnected}
A path-connected closed set $A$ is \emph{locally connected} if, for every $\varepsilon>0$, there exists a $\delta>0$ such that any two points $b$, $c\in A$ with $\abs{b-c}<\delta$ can be joined by a continuum $B\subset A$ of diameter at most $\varepsilon$.

A sequence of closed sets $\{A_n\}_{n=1}^\infty$ is \emph{uniformly locally connected} if, for every $\varepsilon>0$, there exists a $\delta>0$ independent of $n$ such that any two points $b$, $c\in A_n$ with $\abs{b-c}<\delta$ can be joined by a continuum $B_n\subset A_n$ of diameter at most $\varepsilon$.
\end{definition}

\begin{definition} \label{dfn:equicontinuous}
A sequence of maps $\{\Phi_n\}_{n=1}^\infty$ on a domain $\Omega$ is \emph{equicontinuous} if, for all $\varepsilon>0$, there exists a $\delta>0$, depending only on $\varepsilon$, such that if $x$, $y\in \Omega$ and $\abs{x-y}<\delta$, then $\abs{\Phi_n(x)-\Phi_n(y)}<\varepsilon$ for all~$n$.
\end{definition}

In particular, if $\Omega_n$ is part of a uniformly locally connected sequence then $\Omega_n$ is locally connected, and if $\Phi_n$ is part of an equicontinuous sequence then $\Phi_n$ is uniformly continuous.

\section{Convergence results}\label{sec:convergence}

The important result of this section is \thmref{snipesbarton}. This theorem states that if $f$ is the limit of a sequence of harmonic measure distribution functions of well-behaved domains, then $f$ is itself the harmonic measure distribution function of a well-behaved domain. This result may be viewed as a sufficient condition on~$f$ that forces $f=h_\Omega$ for some simply connected domain~$\Omega$.

\begin{theorem} \label{thm:snipesbarton}
Let $f$ be a right-continuous, nondecreasing function. Suppose that there exists a sequence of domains $\{\Omega_n\}_{n=1}^\infty$ such that
\begin{thmenumerate}
\item \label{item:snipesbarton:first} $\Omega_n$ is simply connected,
\item There exist numbers $M$, $\mu>0$ such that for all $n$, $B(0,\mu)\subset \Omega_n\subset B(0,M)$,
\item\label{item:snipesbarton:localconn} $\{\C\setminus\Omega_n\}_{n=1}^\infty$ is uniformly locally connected, and
\item\label{item:snipesbarton:conv} $h_{\Omega_n}\to f$ pointwise at all points of continuity of~$f$.
\end{thmenumerate}
Then there exists a domain $\Omega$ that is bounded and simply connected, contains $0$, and has locally connected complement, such that $f=h_\Omega$. Furthermore, there is a subsequence of $\{\Omega_n\}_{n=1}^\infty$ whose Riemann maps converge uniformly to the Riemann map of~$\Omega$.
\end{theorem}

\thmref{snipesbarton} follows immediately from the following two theorems. (\thmref{sccdconv} will be proven below; \thmref{snipes:conv} was proven in \cite[Theorem~1]{SniW08}.) Let $\D$ denote the unit disk $B(0,1)\subset\C$.

\begin{theorem}\label{thm:sccdconv}
Suppose that $\{\Omega_n\}_{n=1}^\infty$ is a sequence of domains in the complex plane such that the following conditions hold.
\begin{thmenumerate}
\item $\Omega_n$ is simply connected,
\item There exist numbers $M$, $\mu>0$ such that for all $n$, $B(0,\mu)\subset \Omega_n\subset B(0,M)$, and
\item $\{\C\setminus\Omega_n\}_{n=1}^\infty$ is uniformly locally connected.
\end{thmenumerate}
Then the Riemann maps $\Phi_n:\D\mapsto\Omega_n$ are equicontinuous, and so they have continuous extensions to the closed unit disk $\overline\D$. If $\Phi_n(0)=0$, $\Phi_n'(0)>0$ for all~$n$, then the sequence $\{\Phi_n\}_{n=1}^\infty$ contains a subsequence $\{\Phi_{n(k)}\}_{k=1}^\infty$ that converges uniformly on $\overline\D$ to some uniformly continuous map~$\Phi$. Furthermore, the limit $\Phi$ is itself the Riemann map of some domain~$\Omega$ with locally connected complement.
\end{theorem}

\begin{theorem}[Snipes and Ward] \label{thm:snipes:conv}
Let $\Omega$ and $\Omega_n$, $n\geq 1$, be simply connected domains containing the point $z_0=0$, with $\Omega\neq\C$ and $\Omega_n\neq\C$, and with harmonic measure distribution functions $h$ and $h_n$, respectively. Suppose that the normalized Riemann mappings $\Phi:\D\mapsto\Omega$, $\Phi(0)=0$, $\Phi'(0)>0$ and $\Phi_n:\D\mapsto\Omega_n$, $\Phi_n(0)=0$, $\Phi_n'(0)>0$ have continuous extensions to the closed unit disk $\overline\D$. If $\Phi_n\to\Phi$ pointwise on the boundary $\partial\D$, then $h_n\to h$ pointwise at all points of continuity of~$h$.
\end{theorem}
Because we wish to use \thmref{sccdconv}, \thmref{snipes:conv} is an appropriate convergence result to use to prove \thmref{snipesbarton}. However, \thmref{snipes:conv} is not the only known result of its type. As noted in \cite{BetS}, if the domains $\Omega_n$ are simply connected and $\Omega_n\to\Omega$ in the sense of Carath\'eodory, then $h_n(r)\to h(r)$ for almost all~$r$.

To prove \thmref{sccdconv}, we will need the following three theorems (taken from \cite[Theorem~2.1]{Pom}, \cite[Propositon~2.3]{Pom} and \cite[p.~245]{Rud87}, respectively).
\begin{theorem}\label{thm:pommer:localconn}
Let $\Phi$ map $\D$ conformally onto the bounded domain~$\Omega$. Then $\Phi$ has a continuous extension to~$\overline\D$ if and only if $\C\setminus\Omega$ is locally connected.
\end{theorem}

\begin{theorem}\label{thm:pommer:equic}
Let $\Phi_n$ map $\D$ conformally onto $\Omega_n$ with $\Phi_n(0)=0$. Suppose that
\[B(0,\mu)\subset\Omega_n\subset B(0,M)\]
for all $n$. If $\{\C\setminus\Omega_n\}_{n=1}^\infty$ is uniformly locally connected, then $\{\Phi_n\}_{n=1}^\infty$ is equicontinuous on $\overline\D$.
\end{theorem}

\begin{theorem}[Arzela--Ascoli]\label{thm:arzela}
Suppose that $\{\Phi_n\}_{n=1}^\infty$ is a sequence of pointwise bounded, equicontinuous complex functions on $\overline\D$. Then there is a subsequence of the $\Phi_n$ that converges uniformly on~$\overline\D$.
\end{theorem}

\begin{proof}[Proof of \thmref{sccdconv}]
Theorems~\ref{thm:pommer:equic} and~\ref{thm:arzela} imply that a subsequence of $\{\Phi_n\}_{n=1}^\infty$ converges uniformly on $\overline\D$ to some map~$\Phi$.
Because $\{\Phi_n\}_{n=1}^\infty$ is equicontinuous, $\Phi$ must be uniformly continuous. We need only show that $\Phi$ is the Riemann map of some domain~$\Omega$; \thmref{pommer:localconn} will then imply that $\C\setminus\Omega$ is locally connected.

To show that $\Phi$ is a Riemann map, we must show that $\Phi$ is analytic and injective and that $\Phi(\D)$ is open and simply connected.

By \cite[p.~176]{Ahl77}, if $\Phi_n\to \Phi$ uniformly in some domain and $\Phi_n$ is analytic, then $\Phi$ is also analytic and $\Phi_n'\to \Phi'$ uniformly on compact subsets. (This fact may be easily seen from the Cauchy integral formulas.) By \cite[p.~132]{Ahl77}, a nonconstant analytic function maps open sets onto open sets. Also, if $\Phi$ is one-to-one, the preimage of any loop in $\Omega=\Phi(\D)$ must be a loop in $\D$; since $\D$ is simply connected, loops are contractible, and so $\Phi(\D)$ must be simply connected as well.

So we need only show that $\Phi$ is one-to-one. Because $\Phi_n(\D) = \Omega_n\supset B(0,\mu)$ and $\Phi_n\to\Phi$ uniformly, $\Phi$ cannot be a constant. Thus, for any fixed $z\in\D$, the zeros of $\Phi-\Phi(z)$ must be isolated. Let $\gamma$ be any Jordan curve in $\D$ with $z$ in its interior, and with $\abs{\Phi-\Phi(z)}\geq\varepsilon>0$ on~$\gamma$ for some~$\varepsilon$. Consider
\[\frac{1}{2\pi i} \oint_\gamma \frac{\Phi'(\zeta)}{\Phi(\zeta)-\Phi(z)} \,d\zeta.\]
This quantity is equal to the number of zeros of $\Phi-\Phi(z)$ in the interior of $\gamma$. But
\[1=\frac{1}{2\pi i} \oint_\gamma \frac{\Phi_n'(\zeta)}{\Phi_n(\zeta)-\Phi_n(z)} \,d\zeta\]
for all $n$, and so since $\Phi_n\to\Phi$ and $\Phi_n'\to\Phi'$ uniformly on $\gamma$, we must have that
\[1=\frac{1}{2\pi i} \oint_\gamma \frac{\Phi'(\zeta)}{\Phi(\zeta)-\Phi(z)} \,d\zeta.\]
So $\Phi$ is one-to-one.
\end{proof}

\section{Circle domains and a sufficient condition \texorpdfstring{for $f$ to be an $h$-function}{}} \label{sec:sufficient}

In this section, we will provide a more restrictive, but easier to check, condition on a function $f$ that will force it to arise as the harmonic measure distribution function of some domain.

We begin by quoting a known existence result.
\begin{theorem}[{\cite[Theorem~2]{SniW05}}]\label{thm:snipes:circle}
Let $f(r)$ be a right-continuous step function, increasing from~$0$ to~$1$, with its discontinuities at $r_0$, $r_1,\,\dots\,, r_{n}$, where $0<r_0<r_1<\dots<r_{n}$. Then there exists a circle domain $X$ with $n$ arcs whose harmonic measure distribution function $h_X(r)$ is equal to $f(r)$. The radii of the $n$ arcs and of the boundary circle in~$X$ are given by $r_0$, $r_1$, $r_2,\,\ldots\,, r_{n-1}$ and by $r_{n}$ respectively.
\end{theorem}
Monotonic functions can be approximated by step functions. Therefore, if $f$ is a nondecreasing function, then there is some sequence of circle domains $\{X_n\}_{n=1}^\infty$ such that $h_{X_n}\to f$. This sequence does not satisfy the conditions of \thmref{snipesbarton}. However, we will find conditions (\thmref{main}) on the sequence $\{X_n\}_{n=1}^\infty$ such that we may construct a sequence $\{\Omega_n\}_{n=1}^\infty$ of blocked circle domains that does.

Informally, the conditions are that the circle domains $X_n$ do not have too many short arcs, and that the function $f$ increases fast enough.

\thmref{main} provides a sufficient condition for \thmref{snipesbarton} to hold, which may be checked by examining $\{X_n\}_{n=1}^\infty$ rather than all sequences of simply connected domains $\{\Omega_n\}_{n=1}^\infty$ that satisfy $h_{\Omega_n}\to f$.

Our proof of \thmref{main} relies on Lemmas~\ref{lem:deriv}, \ref{lem:radial} and~\ref{lem:unifcctd}. We defer their precise statements and proofs to \secref{circle}.

We begin by fixing some terminology.
Suppose that we have a sequence $\{X_n\}_{n=1}^\infty$ of circle domains, and a sequence $\{\Omega_n\}_{n=1}^\infty$ of blocked circle domains such that $\Omega_n\subset X_n$ and $\partial\Omega_n\setminus\partial X_n$ is a union of gates.
In this section and in \secref{functions} we will use the following symbols to describe such sequences of domains. (In \secref{circle} we will often discuss single domains rather than sequences of domains; when doing so we will use the same notation without the $n$ subscript.)
\begin{itemize}
\item $X_n$ denotes the $n$th circle domain.
\item $\Omega_n$ denotes the $n$th blocked circle domain.
\item $A_{n,k}$ denotes a boundary arc of $X_n$; the innermost arc is $A_{n,0}$, the next innermost arc is $A_{n,1}$, and so on.
\item $r_{n,k}$ denotes the radius of $A_{n,k}$.
\item $\psi_{n,k}$ denotes half the arclength of $A_{n,k}$.
\item $\phi_{n,k}$ denotes the angle (from the real axis) of a gate of~$\Omega_n$ lying between $A_{n,k}$ and $A_{n,k+1}$.
\item $\chi_{n,k}$ denotes the inset angle of that gate.
\item $\eta_{n,j,k}$ measures the depth of the shortest arc between $A_{n,j}$ and $A_{n,k}$.
\item $\theta_{n,j,k}$ measures the depth of the deepest gate between $A_{n,j}$ and $A_{n,k}$.
\end{itemize}
We emphasize that the numbering of boundary arcs and gates is to begin at zero rather than one. See \figref{param} for an illustration of $r_{n,k}$, $\psi_{n,k}$ and $\phi_{n,k}$,~$\chi_{n,k}$.
We define $\chi_{n,k}$ by the equation
\begin{equation}
\label{dfn:chi}
\phi_{n,k}+\chi_{n,k}=\min(\psi_{n,k},\psi_{n,k+1}).\end{equation}
We require that $\phi_{n,k}$ and $\chi_{n,k}$ both be nonnegative; this implies that $\phi_{n,k}\leq\min(\psi_{n,k},\psi_{n,k+1})$ and $\chi_{n,k}\leq\min(\psi_{n,k},\psi_{n,k+1})$.

\begin{figure}
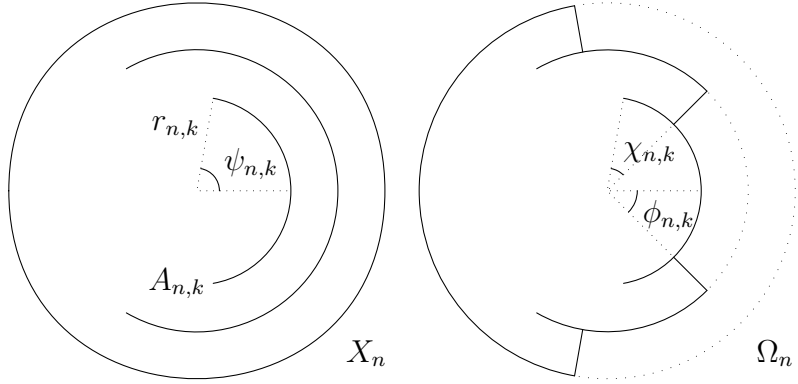

\mypic{2}
\caption{Parameters $r_{n,k}$, $\psi_{n,k}$, and $\phi_{n,k}$, $\chi_{n,k}$, used to describe circle domains $X_n$ and blocked circle domains~$\Omega_n$. Here $n=2$ and $k=0$.}
\label{fig:param}
\end{figure}

We now clarify and make precise the definition of $\eta_{n,j,k}$. Let $A_{n,j}$ and $A_{n,k}$ be any two boundary arcs of~$X_n$. Consider the arcs $A_{n,l}$, $j<l<k$, lying between $A_{n,j}$ and $A_{n,k}$. In many of the theorems to come, we will need to consider the shortest of these arcs. The important value will usually not be the arclength of the shortest arc $A_{n,l}$, but its depth in the channel outlined by $A_{n,j}$ and $A_{n,k}$.
We define
\begin{equation}
\label{dfn:eta}
\eta_{n,j,k}:=\min(\psi_{n,j},\psi_{n,k})-\min_{j\leq l\leq k}(\psi_{n,l}).
\end{equation}
This number is zero if none of the inner arcs are shorter than the outer arcs, and otherwise is half the difference in arclength between the shorter of the outer arcs and the shortest of the inner arcs. (We divide the arc length difference by two because there are two ends to the channel and the shortest arc is inset in both of them.)

Similarly,
\begin{equation}
\label{dfn:theta}
\theta_{n,j,k}:=\min(\psi_{n,j},\psi_{n,k})-\min_{j\leq l< k}(\phi_{n,l})
\end{equation}
measures the depth of the deepest gate in $\Omega_n$ between $A_{n,j}$ and~$A_{n,k}$.
See \figref{moreparam} for an illustration of $\eta_{n,j,k}$ and $\theta_{n,j,k}$.

We observe that we can bound $\theta_{n,j,k}$ by $\eta_{n,j,k}$ and $\chi_{n,m}$ for $j\leq m<k$.
By \bartoneqref{formula}{dfn:chi}, if $j\leq m<k$ then
\[\phi_{n,m}=\min(\psi_{n,m},\psi_{n,m+1})-\chi_{n,m}
\geq \min_{j\leq l\leq k} \psi_{n,l}-\chi_{n,m}
\geq \min_{j\leq l\leq k} \psi_{n,l}-\max_{j\leq l<k}\chi_{n,l}.\]
By \bartoneqref{formula}{dfn:eta},
\begin{align*}
\eta_{n,j,k}
&=\min(\psi_{n,j},\psi_{n,k})-\min_{j\leq l\leq k}\psi_{n,l}
\\&=\min(\psi_{n,j},\psi_{n,k})-\min_{j\leq l\leq k}\psi_{n,l}
+\max_{j\leq l<k}\chi_{n,l}
-\max_{j\leq l<k}\chi_{n,l}.
\end{align*}
Combining these formulas we have that
\[\eta_{n,j,k}\geq
\min(\psi_{n,j},\psi_{n,k})
-\min_{j\leq m\leq k}\phi_{n,m}
-\max_{j\leq l<k}\chi_{n,l}.\]
By \bartoneqref{formula}{dfn:theta}, the right-hand side is equal to $\theta_{n,j,k}-\max_{j\leq l<k}\chi_{n,l}$, and so
\begin{equation}
\label{eqn:thetaeta}
\theta_{n,j,k}
\leq \eta_{n,j,k} + \max_{j\leq l<k}\chi_{n,l}.
\end{equation}

\begin{figure}
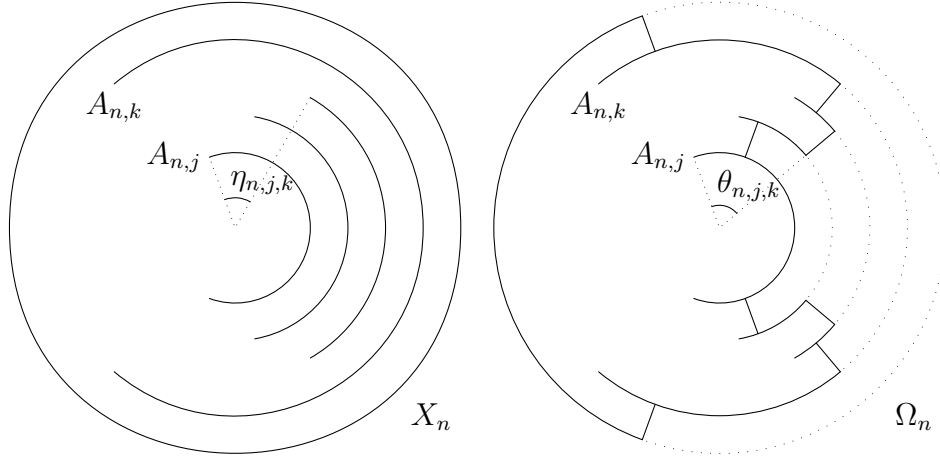

\mypic{7}
\caption{The angles $\theta_{n,j,k}$ and $\eta_{n,j,k}$ used to describe circle domains $X_n$ and blocked circle domains~$\Omega_n$. Here $n=4$, $j=0$, and $k=3$.}
\label{fig:moreparam}
\end{figure}

We now consider functions. Suppose that $f$ is a \emph{candidate for a harmonic measure distribution function}, meaning that $f$ is a right-continuous function defined on $\R_+$ that is 0 on $(0,\mu)$, is nondecreasing on $[\mu,M]$, is 1 on $[M,\infty)$, and satisfies $0<f<1$ on $(\mu,M)$, for some numbers $0<\mu<M$.

Define the \emph{minimal secant slope} $\alpha$ of $f$ by
\begin{equation}
\label{eqn:derivmin}
\alpha:=(M-\mu)\inf\biggl\{\frac{f(\rho_2)-f(\rho_1)}{\rho_2-\rho_1}
\biggm|\mu\leq\rho_1<\rho_2\leq M\biggr\}.
\end{equation}
The infimum in this definition is the infimum of the slopes of secant lines to the graph of~$f$ between $\mu$ and~$M$; multiplying by $(M-\mu)$ normalizes this number so that $0\leq \alpha\leq 1$.

We fix
\[r_{n,k} := \mu + (M-\mu)\frac{k}{n}\quad \text{for $0\leq k\leq n$,}\]
 so that $r_{n,0}=\mu$ and $r_{n,n}=M$.
Define $f_n(r)$ by
\[f_n(r):=\begin{cases}
0, & 0<r<r_{n,0};\\
f(r_{n,k}), & r_{n,k}\leq r <r_{n,k+1};\\
1, & r_{n,n}\leq r<\infty.\end{cases}\]
Thus, $f_n$ is a right-continuous nondecreasing step function and approximates $f$ from below. Let $X_n$ be a circle domain with boundary arcs at radii $r_{n,k}$ and with $h_{X_n}=f_n$; by \thmref{snipes:circle} such an $X_n$ exists. Then $h_{X_n}=f_n\to f$ pointwise at all points of continuity of~$f$. We remark that $\psi_{n,n}=\pi$, that $\psi_{n,0}=0$ if $f$ is continuous at~$\mu$, and that if $\alpha>0$ then $\psi_{n,k}>0$ for all $k>0$.

We now show that if the domains $X_n$ do not have too many short arcs, and if the function $f$ increases fast enough, then $f$ is the harmonic measure distribution function of some bounded, simply connected domain~$\Omega$.

\begin{theorem}\label{thm:main} Suppose that the function $f$ is a candidate for a harmonic measure distribution function. Let $\alpha$ be the minimal secant slope given by \bartoneqref{equation}{eqn:derivmin}, and define the circle domains $X_n$ as above.

Let $\{\kappa_n\}_{n=1}^\infty$ be any sequence of numbers, for example $\kappa_n=((M-\mu)/(\mu n)) \log n $, such that
\begin{enumerate}
\item \label{item:kappasmall} $\displaystyle\lim_{n\to\infty} \kappa_n=0$, and
\item \label{item:kappabig} $\displaystyle\lim_{n\to\infty} n\exp\left(-\frac{\pi\mu}{2(M-\mu)}n\kappa_n\right)= 0.$
\end{enumerate}
Suppose that
\begin{enumerate}
\addtocounter{enumi}{2}
\item \label{item:derivpos} the minimal secant slope $\alpha>0$, and
\item \label{item:fewshortarcs} ${\displaystyle\lim_{n\to\infty} }\sigma_n/\sqrt n= 0$, where $\sigma_n$ is the number of boundary arcs of $X_n$ with arclength at most $2\kappa_n$.
\end{enumerate}
Then there exists a sequence $\{\Omega_n\}_{n=1}^\infty$ of blocked circle domains, satisfying the conditions \eqref{item:snipesbarton:first}--\eqref{item:snipesbarton:localconn} of \thmref{snipesbarton}, such that $h_{\Omega_n}\to f$ pointwise at all points of continuity of~$f$. So by \thmref{snipesbarton}, $f$ is the harmonic measure distribution function of some bounded, simply connected domain~$\Omega$.
\end{theorem}

\begin{proof} For each domain $X_n$, we will construct a suitable blocked circle domain $\Omega_n\subset X_n$. By construction, $h_{X_n}\to f$ pointwise at points of continuity of~$f$. We want to show that if $\lim_{n\to\infty} \sigma_n/\sqrt n= 0$ and $\alpha>0$, then we may choose $\Omega_n$ such that $\{\Omega_n\}_{n=1}^\infty$ and~$f$ satisfy the conditions of \thmref{snipesbarton}. The only difficult conditions are \eqref{item:snipesbarton:localconn} and \eqref{item:snipesbarton:conv}, that is, the requirement that $\{\C\setminus\Omega_n\}_{n=1}^\infty$ be uniformly locally connected and the requirement that $h_{\Omega_n}\to f$. Since $h_{X_n}\to f$, we may replace \eqref{item:snipesbarton:conv} by a requirement that $h_{X_n}-h_{\Omega_n}\to 0$.

Recall that $r_{n,k}=\mu+k(M-\mu)/n$. The circle domains $X_n$, and thus $\psi_{n,k}$ and $\eta_{n,j,k}$, are then determined by the fact that $h_{X_n}=f_n$. Given the inset angles $\chi_{n,k}$ of the gates, the domains $\Omega_n$ and the parameters $\phi_{n,k}$, $\theta_{n,j,k}$ are determined.

We are free to choose $\chi_{n,k}$ provided $0\leq \chi_{n,k}\leq \min(\psi_{n,k},\psi_{n,k+1})$. We set the inset angles $\chi_{n,k}$ of the gates in $\Omega_{n,k}$ to be
\[\chi_{n,k}=\min(\psi_{n,k},\psi_{n,k+1},\kappa_n).\]
The number $\kappa_n$ thus determines the size and shape of (most of) the ends of the channels. Since $\kappa_n\to 0$, the channel ends must eventually become short, but since $n\kappa_n \to \infty$ and the width of the channels is $(M-\mu)/n$, the channel ends must become thin much faster than they become short.

We remark that if the inset angle $\chi_{n,k}\neq \kappa_n$, then $\chi_{n,k}=\min(\psi_{n,k},\psi_{n,k+1})$. Recall that $\chi_{n,k}+\phi_{n,k}=\min(\psi_{n,k},\psi_{n,k+1})$, where $\phi_{n,k}$ is the gate angle. Therefore, if $\chi_{n,k}\neq \kappa_n$ then $\phi_{n,k}=0$; there is only one gate between $A_{n,k}$ and $A_{n,k+1}$, and it lies along the positive real axis.

Furthermore, if $\phi_{n,k}=0$ then $\chi_{n,k}=\min(\psi_{n,k},\psi_{n,k+1})$. Since $\chi_{n,k}\leq \kappa_n$ this implies that $\psi_{n,k}\leq\kappa_n$ or $\psi_{n,k+1}\leq \kappa_n$; thus, because there are at most $\sigma_n$ arcs $A_k$ with $\psi_{n,k}\leq\kappa_n$, there are at most $2\sigma_n$ gates that lie along the real axis, one on each side of each short arc.

Recall that we imposed two conditions on $\{\kappa_n\}_{n=1}^\infty$. The reasons are as follows. A necessary (not sufficient!) condition for uniform local connectivity of the sequence $\{\Omega_n\}_{n=1}^\infty$ is that the inset angle $\chi_{n,k}$ of the gates be small whenever $r_{n,k+1}-r_{n,k}$ is small. (See \lemref{unifcctd}.) Since $\chi_{n,k}=\min(\psi_{n,k},\psi_{n,k+1},\kappa_n)$, and $\psi_{n,k}$ need not be small, we must have that $\kappa_n\to 0$ as $n\to\infty$.

However, $h_{X_n}-h_{\Omega_n}$ is controlled by the harmonic measure of the gates (see \lemref{radial}), which is small if $\chi_{n,k}$ is large compared with $r_{n,k+1}-r_{n,k}$ (see \lemref{curvechannel}); thus, while $\kappa_n\to 0$, $\kappa_n$ cannot go to zero too quickly.

We make our requirements precise.
In \lemref{radial}, we will establish a bound on $\abs{h_{X_n}-h_{\Omega_n}}$; it is
\[
\abs{h_{X_n}(r)-h_{\Omega_n}(r)}
\leq
\frac{32}{\pi}\sum_{k:\phi_{n,k}>0}
\exp\left(-\frac{\pi\, r_{n,k}\, \chi_{n,k}}{2(r_{n,k+1}-r_{n,k})}\right)
+\frac{2}{\pi}\sum_{k:\phi_{n,k}=0} \sqrt{\frac{r_{n,k+1}-r_{n,k}}{r_{n,k}}}
.\]

Recall that there are at most $2\sigma_n$ gates that lie along the real axis; thus, there  are at most $2\sigma_n$ numbers $k$ that satisfy $\phi_{n,k}=0$. Since $r_{n,k+1}-r_{n,k}=(M-\mu)/n$, and $\mu\leq r_{n,k}$, the second sum is at most $2\sigma_n \sqrt{(M-\mu)/({n\mu})}$.

There are $n$ boundary arcs (and so at most $n$ pairs of gates with $\phi_{n,k}>0$). Recall that if $\phi_{n,k}>0$ then $\chi_{n,k}=\kappa_n$. Again by our choice of $r_{n,k}$, the first sum is at most $n\exp(-(\pi \,\mu \,\kappa_n)/(2(M-\mu)/n))$.

So
\[\abs{h_{X_n}(r)-h_{\Omega_n}(r)}\leq
\frac{32}{\pi}n
\exp\left(-\frac{\pi \mu}{2(M-\mu)}n \kappa_n\right)
+\frac{4}{\pi}\sigma_n \sqrt{\frac{M-\mu}{n\mu}}
.\]

By our assumptions on $\sigma_n$ and $\kappa_n$, both terms go to zero, and so \bartoneqref{condition}{item:snipesbarton:conv} holds: $h_{\Omega_n}\to f$ at points of continuity of~$f$.

We turn to \bartoneqref{condition}{item:snipesbarton:localconn}. Assume that the minimal secant slope $\alpha$ is greater than~$0$. Recall the definitions of $\eta_{n,j,k}$ and $\theta_{n,j,k}$ (the depths of the shortest arc and deepest gate, respectively, between $A_{n,j}$ and $A_{n,k}$; see \figref{moreparam}).

\bartoneqref{Equation}{eqn:thetaeta} states that $\theta_{n,j,k}\leq\eta_{n,j,k}+\max_{j\leq l<k}\chi_{n,l}$. Since $\chi_{n,l}\leq\kappa_n$ for all $n$ and~$l$, this implies that
\[\theta_{n,j,k}\leq \eta_{n,j,k}+\kappa_n.\]

In \lemref{unifcctd}, we will establish sufficient conditions for $\{\C\setminus\Omega_n\}_{n=1}^\infty$ to be uniformly locally connected. These conditions are that, for every $\varepsilon>0$, there exist positive numbers $\delta_1(\varepsilon)$, $\delta_2(\varepsilon)$ such that
\begin{itemize}
\item If $0\leq k\leq n$ and $\pi-\psi_{n,k}<\delta_2(\varepsilon)$, then $M-r_{n,k}<\varepsilon$, and
\item If $0\leq j<k\leq n$ and ${r_{n,k}-r_{n,j}}<\delta_1(\varepsilon)$ then $\theta_{n,j,k}<\varepsilon$.
\end{itemize}
Informally, these conditions say that if two arcs in a blocked circle domain are sufficiently close, then the gates between them are not too deep, and that if an arc is long enough to be almost a full circle, then it is close to the outer boundary circle.

If $0<j\leq n$, then $\omega(0,A_{n,j},X_n)=f(r_{n,j})-f(r_{n,j-1})\geq \alpha/n$. When $\alpha>0$, we can use this lower bound on the harmonic measure of each arc to ensure that the conditions of \lemref{unifcctd} hold.

We do so as follows. In \lemref{deriv}, we will show that if $\alpha>0$, then for each $\varepsilon>0$ there exist numbers $\delta_1'(\varepsilon)$, $\delta_2'(\varepsilon)>0$ (depending on $\alpha$, $M$,~$\mu$) such that
\begin{itemize}
\item If $0\leq k\leq n$ and $\pi-\psi_{n,k}<\delta_2'(\varepsilon)$, then $M-r_{n,k}<\varepsilon$, and
\item If $0\leq j<k\leq n$ and $r_{n,k}-r_{n,j}<\delta_1'(\varepsilon)$, then $\eta_{n,j,k}<\varepsilon$.
\end{itemize}
Fix $\varepsilon>0$. Let $\delta_2(\varepsilon)=\delta_2'(\varepsilon)$. Since $\kappa_n\to 0$ as $n\to\infty$, there is some $N_\kappa(\varepsilon)$ such that if $n>N_\kappa(\varepsilon)$ then $\kappa_n<\varepsilon/2$. Let
\[\delta_1(\varepsilon) = \min\left(\delta_1'(\varepsilon/2),\frac{M-\mu}{N_\kappa(\varepsilon)}\right).\]
Suppose $j<k$ and $r_{n,k}-r_{n,j}<\delta_1(\varepsilon)$. Then because $\delta_1(\varepsilon)\leq \delta_1'(\varepsilon/2)$, the depth $\eta_{n,j,k}$ of the shortest arc between $A_j$ and $A_k$ satisfies $\eta_{n,j,k}<\varepsilon/2$. Furthermore, since $j\neq k$ we have that
\[\frac{M-\mu}{n}\leq r_{n,k}-r_{n,j}<\delta_1(\varepsilon)\leq \frac{M-\mu}{N_\kappa(\varepsilon)},\]
and so $N_\kappa(\varepsilon)<n$; thus $\kappa_n<\varepsilon/2$. Then $\theta_{n,j,k}\leq \eta_{n,j,k}+\kappa_n\leq\varepsilon$.

So if $\alpha>0$ then the conditions of \lemref{unifcctd} hold, as desired.
\end{proof}

\section{An easy-to-check sufficient condition \texorpdfstring{for $f$ to be an $h$-function}{}} \label{sec:functions}

In this section, we exhibit a family of functions that satisfy the conditions of \thmref{main}, and are thus the harmonic measure distribution functions of simply connected bounded domains (\thmref{fjump}). Loosely speaking, these are functions that begin with a jump and reach~1 quickly after that.

Our proof of \thmref{fjump} relies on Lemmas~\ref{lem:fjump} and~\ref{lem:deriv}; we defer their statements and proofs to \secref{circle}.

\begin{theorem}\label{thm:fjump}
Suppose that $f$ is right-continuous, $0$ on $[0,\mu)$, strictly increasing on $[\mu, M]$, and $1$ on $[M,\infty)$, with $0<\mu<M$ and $0<f<1$ on $(\mu,M)$.
Let $\alpha$ be the minimal secant slope of $f$ given by \bartoneqref{equation}{eqn:derivmin},  so $f(\rho_2)-f(\rho_1)\geq \alpha(\rho_2-\rho_1)/(M-\mu)$ for all $\rho_1$,~$\rho_2$ with $\mu\leq \rho_1<\rho_2\leq M$.
Let $\beta=f(\mu)=\lim_{r\to\mu^+} f(r)$.
Notice that $0\leq\beta\leq 1$ and $0\leq\alpha\leq 1-\beta$.

If $\alpha>0$ and $\beta>0$, then there is a number $m_0>0$ depending only on $\alpha$ and $\beta$ such that, if $(M-\mu)/\mu<m_0$, then $f$ satisfies the conditions of \thmref{main}. Hence there exists a simply connected bounded domain~$\Omega$, arising from a sequence of circle domains, such that $f=h_\Omega$.

Furthermore, $m_0\geq \min(m_1,m_2,m_3)$, where
\begin{align}
&m_1 := \frac{1}{e-1},\label{eqn:m1}\\
&m_2 := \frac{\pi^2}{8\log (256/\pi\alpha)},\label{eqn:m2}\\
&
\frac{2}{\pi}m_3\left(2\log\left(1+\frac{1}{m_3}\right)+\pi^2\right)
+\frac{4}{\pi} m_3\log\left(\frac{256}{\pi\alpha}\right)
=\pi\beta
.\label{eqn:m3}
\end{align}
\end{theorem}

We remark that for each fixed $\alpha$, $\beta\in (0,1)$, the function
\[g(m)=
\frac{2}{\pi}m\left(2\log\left(1+\frac{1}{m}\right)+\pi^2\right)
+\frac{4}{\pi} m\log\left(\frac{256}{\pi\alpha}\right)
\]
is continuous and strictly increasing on $(0,\infty)$, and satisfies $\lim_{m\to 0^+} g(m)=0$, $\pi\beta < \pi < g(1)$; thus, there is a unique (small) positive number $m_3$ that solves \bartoneqref{equation}{eqn:m3}.

\begin{proof} Choose some such function~$f$.
Let $X_n$, $A_{n,k}$, $\psi_{n,k}$ be as in \secref{sufficient}.

By our construction of $X_n$, the harmonic measure of the innermost arc $A_{n,0}$ of $\partial X_n$ is given by \[\omega(0,A_{n,0},X_n)=f(r_{n,0})=f(\mu)=\beta\]
for all~$n$. So by \lemref{fjump}, if $\mu\geq M(1-1/e)$, then
\[\psi_{n,0}>\min\left(\frac{\pi}{2},
\pi\beta -\frac{2}{\pi} \frac{M-\mu}{\mu}
\left(2\log\left(\frac{M}{M-\mu}\right)+\pi^2
\right)\right).
\]

For fixed $n$ and $k$, $A_{n,k}$ is a boundary arc of $X_n$ that lies between $A_{n,0}$ and $A_{n,n}=\partial B(0,M)$. Thus, by \lemref{deriv}, we have that
\[\psi_{n,0}-\psi_{n,k}
\leq \eta_{n,0,n} \leq \frac{4}{\pi} \frac{M-\mu}{\mu}
\log\left(\frac{256}{\pi\alpha}\right).
\]

Thus, if $\mu>M(1-1/e)$ and $0\leq k\leq n$, then
\[
\psi_{n,k}\geq
\min\left(\frac{\pi}{2},
\pi\beta -\frac{2}{\pi}\frac{M-\mu}{\mu}
\left(2\log\left(\frac{M}{M-\mu}\right)+\pi^2
\right)\right)
-
\frac{4}{\pi} \frac{M-\mu}{\mu}
\log\left(\frac{256}{\pi\alpha}\right)
.\]

If $(M-\mu)/\mu<m_1$, then $\mu>M(1-1/e)$. If $(M-\mu)/\mu < m_2$, then
\[
\frac{\pi}{2}-
\frac{4}{\pi}
\frac{M-\mu}{\mu}
\log\left(\frac{256}{\pi\alpha}
\right)
>0.\]

Suppose that $(M-\mu)/\mu<m_3$; by monotonicity of $g(m)$ we have that
\[
\frac{2}{\pi}\frac{M-\mu}{\mu}\left(2\log\left(\frac{M}{M-\mu}\right)+\pi^2\right)
+\frac{4}{\pi} \frac{M-\mu}{\mu}\log\left(\frac{256}{\pi\alpha}\right)
<\pi\beta.
\]
Thus, if $(M-\mu)/\mu<\min(m_1,m_2,m_3)$, then there is some positive constant $\psi$ such that $\psi_{n,k}\geq\psi$ for all $n$ and~$k$; that is, the arclength of every boundary arc is at least~$2\psi$.

Let $\{\kappa_n\}_{n=1}^\infty$ be any sequence that satisfies Conditions~\eqref{item:kappasmall} and~\eqref{item:kappabig} of \thmref{main}. By assumption on $f$, Condition~\eqref{item:derivpos} holds. Since $\lim_{n\to\infty}\kappa_n=0$, if $n$ is large enough then \emph{no} boundary arcs of $X_n$ have arclength less than $2\kappa_n$, and so Condition~\eqref{item:fewshortarcs} holds.

Thus the conditions of \thmref{main} hold, and so a bounded simply connected domain $\Omega$ exists such that $f=h_\Omega$.
\end{proof}

\begin{remark}\label{rmk:fjumpexample} If $\alpha=\beta=\frac{1}{2}$, and the numbers $m_l$ are defined as above, then $m_1\approx 0.58198$, $m_2\approx 0.24220$, and $m_3\approx 0.09922$. Recall the function
\[f(r)=\begin{cases} 0, & 0< r\leq 1\\
\displaystyle
\frac{1}{2}+\frac{1}{2}\frac{r-1}{0.0992}, & 1\leq r \leq 1.0992\\
1, & 1.0992\leq r
\end{cases}\]
of \bartoneqref{equation}{eqn:jumpfunction}. We remark that for this function, $\alpha=\beta=\frac{1}{2}$. Also, $\mu=1$ and $M=1.0992$, so $(M-\mu)/\mu=M-1=0.0992<\min(m_1,m_2,m_3)$. Thus, we see that $f$ satisfies the conditions of \thmref{fjump}, and so $f=h_\Omega$ for some simply connected domain~$\Omega$.
\end{remark}

\section{Uniqueness of the domain \texorpdfstring{$\Omega$}{}}\label{sec:unique}

Let $f$ be a function. Under certain conditions (Theorems~\ref{thm:main} and~\ref{thm:fjump}), there exists a simply connected domain~$\Omega$ such that $f=h_\Omega$. We are interested in whether this domain $\Omega$ is unique; that is, if $\Omega$ and $\widetilde\Omega$ are domains and $h_\Omega=h_{\widetilde\Omega}$, what else must be true of $\Omega$ and $\widetilde\Omega$ to allow us to conclude that $\Omega=\widetilde\Omega$?

It is clear that some conditions must be imposed. If a domain~$\Omega$ is rotated around the point~0, or reflected across a line through~0, then the domain is changed but the harmonic measure distribution function remains the same.

If a single point is deleted from a domain $\Omega$, then the harmonic measure of all boundary sets in $\partial\Omega$ is unchanged, and so the harmonic measure distribution function is unchanged. More generally, if $E\subset\Omega$ has harmonic capacity zero and $0\notin E$, then $\Omega$ and $\Omega\setminus E$ have the same harmonic measure distribution function. In the following discussion, we will disregard sets of harmonic capacity zero.

Much more interesting examples of non-uniqueness exist.
Consider domains whose harmonic measure distribution functions are step functions. (Such functions were studied extensively in \cite{SniW05}. The underlying domains of course are not simply connected.) If $h_{X}$ is a step function with discontinuities at $r_0$, $r_1,\dots,r_n$, then $\partial X$ is a subset of $\cup_{k=0}^n \partial B(0,r_k)$. The proof of \cite[Theorem~2]{SniW05} implies that there are uncountably many such domains $X$ with $h_X=f$. For example, $\partial X\cap \partial B(0,r_n)$ may be taken to be an arc with  arclength of any preassigned number $\psi$, $0<\psi\leq 2\pi$; the sets $\partial X \cap \partial B(0,r_k)$ may be taken to be connected arcs centered at any preassigned angles, or indeed to be disconnected sets.

Thus, if $h_X$ is a step function, then $h_X=h_{\widetilde X}$ for many domains~${\widetilde X}\neq X$.

However, if $X$ and ${\widetilde X}$ are circle domains in the sense of \dfnref{circle}, then $h_X=h_{\widetilde X}$ implies $X={\widetilde X}$. (See \lemref{circleunique}.)
Furthermore, if $h_X$ is a step function, then $X$ is a circle domain if and only if $X$ is bounded and symmetric in the sense of \dfnref{symmetric}.

We conjecture that these conditions suffice to imply uniqueness.
\begin{conjecture}
\label{conj:unique}
Suppose that $\Omega$ and $\widetilde\Omega$ are two domains, both of which are bounded and symmetric in the sense of \dfnref{symmetric}. Suppose further that $h_\Omega=h_{\widetilde\Omega}$. Then $\Omega=\widetilde\Omega$ up to a set of harmonic capacity zero.
\end{conjecture}

We will prove \conjref{unique} only in the two special cases where $\Omega$ is a circle domain (\lemref{circleunique}) or where $\Omega$ is simply connected (\thmref{unique}). We will also show (\lemref{mainsymmetric}) that the simply connected domains $\Omega$ produced by \thmref{main} are symmetric; since they are clearly bounded, \thmref{unique} will apply. Thus, if a function satisfies the conditions of \thmref{main}, then it arises as the $h$-function of a unique bounded, simply connected symmetric domain.
In \thmref{mainunique}, we will state this conclusion more precisely.

Throughout this section, by ``symmetric'' we mean ``symmetric in the sense of \dfnref{symmetric}.''

\begin{lemma}\label{lem:circleunique} Let $X$ and $\widetilde X$ be two circle domains. Suppose that $h_X=h_{\widetilde X}$. Then $X=\widetilde X$ except possibly for finitely many points on the positive real axis.\end{lemma}

\begin{proof} In the proof of \thmref{main}, it was convenient to allow finitely many boundary arcs of arclength~0, that is, boundary ``arcs'' consisting of single points on the positive real axis. Such points have harmonic measure zero and cannot be detected from the harmonic measure distribution function. For the remainder of this proof, we ignore such points; that is, we assume that the boundary arcs of $X$ and $\widetilde X$ have positive arclength.

Since $X$ is a circle domain, $h_X$ must be a step function. Let its discontinuities be at $r_0$, $r_1,\dots,r_n$. The boundary arcs of $X$ are circular arcs with midpoints lying on the positive real axis. Let $A_k$ and $\widetilde A_k$ be the boundary arcs of $X$ and $\widetilde X$, respectively, and let $\psi_k$, $\widetilde \psi_k$ be half their arclengths.

We need only show $\psi_k=\widetilde\psi_k$ for $0\leq k\leq n$. Since $X$ and $\widetilde X$ are bounded, the outermost boundary component is a full circle, and so $\psi_n=\widetilde\psi_n = \pi$. Suppose that $\psi_k>\widetilde \psi_k$ for at least one~$k$. Let $E$ be the union of all arcs $A_k$ such that $\psi_k>\widetilde\psi_k$, and let $\widetilde E$ be the union of the corresponding arcs of~$\widetilde X$. Then $\widetilde E$ is nonempty and $\widetilde E\subsetneq E$, and since $A_n\not\subset E$, we have that $\widetilde E\subsetneq \partial\widetilde X$.

Then
$\omega(0,\widetilde E, \widetilde X)=\omega(0,E,X) $ because $h_X=h_{\widetilde X}$. But $\omega(0,E,X)\geq \omega(0,E, X\cap\widetilde X) $ by the property of monotonicity in the domain of harmonic measure.
Since $E\cap\partial\widetilde X$ is a \emph{proper} subset of both $E$ and~$\partial\widetilde X$, we have that $\omega(0,E, \widetilde X\setminus E)>\omega(0, E\cap\partial\widetilde X, \widetilde X)$.

But $E\cap\partial\widetilde X=\widetilde E$ and $X\cap\widetilde X = \widetilde X\setminus E$.
Thus,
\[\omega(0,\widetilde E, \widetilde X)=\omega(0,E,X) \geq \omega(0,E, X\cap\widetilde X) >\omega(0,\widetilde E, \widetilde X).\]
This is a contradiction; thus $\psi_k\leq \widetilde \psi_k$ for $0\leq k\leq n$. Similarly, $\widetilde \psi_k\leq \psi_k$ for $0\leq k\leq n$ and so $X=\widetilde X$ except possibly for finitely many points on the real axis.
\end{proof}

Now, we consider simply connected domains. We begin with some general remarks.
\begin{remark}\label{rmk:sccdsymm} Let $\Omega$ be a bounded simply connected symmetric domain. Let $\mu$ and $M$ be the largest and smallest numbers, respectively, such that $B(0,\mu)\subset\Omega\subset B(0,M)$. By the remarks after \dfnref{symmetric}, $\Omega\cap\R=(-M,\mu)$.

Let $\Phi:\D\mapsto\Omega$ be the Riemann map;  assume that $\Phi$ is normalized such that $\Phi(0)=0$ and $\Phi'(0)>0$. Since normalized Riemann maps are unique, and since $\Omega$ is symmetric about the real axis, $\Phi(z)=\overline{\Phi(\overline z)}$. Thus, $\Phi(z)$ is real if and only if $z$ is real. Since $\Phi(0)=0$, $\Phi'(0)>0$, and $\Phi$ is continuous and one-to-one, we have that $\Phi((-1,0))=(-M,0)$ and $\Phi((0,1))=(0,\mu)$.
\end{remark}

\begin{lemma} \label{lem:mainsymmetric} Let $\{\Omega_n\}_{n=1}^\infty$ be a sequence of bounded, simply connected domains containing 0 that are symmetric in the sense of \dfnref{symmetric}. Let $\Phi_n:\D\mapsto\Omega_n$ be the Riemann maps of the domains $\Omega_n$, normalized so that $\Phi_n(0)=0$ and $\Phi_n'(0)>0$. Suppose that $\Phi_n\to\Phi$ uniformly on~$\D$, where $\Phi$ is the Riemann map of some simply connected domain~$\Omega$.

Then $\Omega$ is symmetric. In particular, if $f$ satisfies the conditions of \thmref{main} then $f=h_\Omega$ for some bounded, simply connected, symmetric domain~$\Omega$.
\end{lemma}

\begin{proof}
Recall that a bounded domain $\Omega$ is symmetric in the sense of \dfnref{symmetric} if and only if $\Omega\cap \partial B(0,r)$ is path-connected for all $r>0$, $\Omega$ is symmetric about the real axis, and $(-M,0)\subset \Omega\subset B(0,M)$ for some $M>0$.

By \rmkref{sccdsymm}, $\Phi_n(\overline z)=\overline{\Phi_n(z)}$ for all $z\in\D$. Since $\Phi(z)=\lim_{n\to\infty} \Phi_n(z)$, we have that $\Phi(z)=\overline{\Phi(\overline z)}$, implying that $\Omega$ is also symmetric about the real axis.

Let $M_n$, $M$ be the smallest numbers such that $\Omega_n\subset B(0,M_n)$, $\Omega\subset B(0,M)$. Since $\Phi_n\to\Phi$ uniformly, $M_n\to M$. By \rmkref{sccdsymm}, $\Phi_n((-1,0))=(-M_n,0)$, and so $\Phi((-1,0))=(-M,0)$. Thus $(-M,0)\subset\Omega\subset B(0,M)$.

We need only show that the set $\Omega\cap \partial B(0,r)$ is path-connected for all $0<r<M$. We remark that it is easier to show path-connectedness of open sets than of arbitrary sets. Thus, we begin by showing that if the two open sets $\Omega\cap B(0,r)$ and $\Omega\setminus\overline{B(0,r)}$ are path-connected, then so is $\Omega\cap \partial B(0,r)$.

Let $z$, $w\in \Omega\cap \partial B(0,r)$. Then $\abs{z}=\abs{w}=r$. We want to show that some path-connected arc of $\partial B(0,r)$, with endpoints $z$ and~$w$, is contained in~$\Omega$.

Because $\Omega$ is open, there exists some $\varepsilon>0$ such that $B(z,\varepsilon)\subset\Omega$ and $B(w,\varepsilon)\subset\Omega$. Take $z_-\in B(z,\varepsilon)\cap B(0,r)$, and take $w_-\in B(w,\varepsilon) \cap B(0,r)$. Suppose that $\Omega\cap B(0,r)$ is path-connected. Then there exists a path $\gamma_-$ that connects $z$ to $z_-$ to $w_-$ to $w$, and that (except for the endpoints) lies entirely in $\Omega\cap B(0,r)$.

Similarly, if $\Omega\setminus\overline{B(0,r)}$ is path-connected, then there is a path $\gamma_+$ that connects $w$ to $z$ and that lies in $\Omega\setminus\overline{B(0,r)}$. Then $\gamma_+\cup\gamma_-$ is a simple closed curve lying in~$\Omega$. Because $\Omega$ is simply connected, the interior of $\gamma_+\cup\gamma_-$ must lie in $\Omega$, and so an arc of $\partial B(0,r)$ connecting $z$ and $w$ must lie in~$\Omega$.

Thus, if the two open sets $\Omega\cap B(0,r)$ and $\Omega\setminus \overline {B(0,r)}$ are both path-connected, then so is the set $\Omega\cap \partial B(0,r)$.

We wish to show that $\Omega\cap B(0,r)$ is path-connected. We first show that $\Omega_n\cap B(0,\rho)$ is path-connected for all $n\geq 1$ and all $\rho>0$.
Choose some $n\geq 1$ and fix some $\rho>0$. Since $\Omega_n$ is symmetric, if $z$, $w\in \Omega_n$ with $\abs{z}>\abs{w}$, then $z$ and $w$ may be connected by a path lying along the two arcs $\partial B(0,\abs{z})$ and $\partial B(0,\abs{w})$ and along the segment $[-\abs{z},-\abs{w}]$ lying on the negative real axis. In particular, if $\abs{z}<\rho$ and $\abs{w}<\rho$ then $z$ and $w$ may be connected by a path lying in $\Omega_n\cap B(0,\rho)$.

We now pass to the domain $\Omega$.
Let $z_0$, $z_1\in\Omega\cap B(0,r)$. Then for $k=0$, $1$, we have $z_k=\Phi(\zeta_k)$ for some $\zeta_0$, $\zeta_1\in\D$. Let $\varepsilon=\min(r-\abs{z_0},r-\abs{z_1})/2$, and let $n$ be large enough that $\abs{\Phi_n(\zeta)-\Phi(\zeta)}<\varepsilon$ for all $\zeta\in\D$.
So
\[\abs{\Phi_n(\zeta_k)}
\leq\abs{\Phi(\zeta_k)}+\abs{\Phi_n(\zeta_k)-\Phi(\zeta_k)}
<\abs{z_k}+\varepsilon
\leq r-\varepsilon,\]
and thus $\Phi_n(\zeta_k)\in \Omega_n\cap B(0,r-\varepsilon)$. Choosing $\rho=r-\varepsilon$, we see that $\Omega_n\cap B(0,r-\varepsilon)$ is path-connected, and so there is some continuous function $\gamma_n:[0,1]\mapsto\Omega_n\cap B(0,r-\varepsilon)$ such that  $\gamma_n(0)=\Phi_n(\zeta_0)$ and $\gamma_n(1)=\Phi_n(\zeta_1)$. Consider $\gamma(t)=\Phi(\Phi_n^{-1}(\gamma_n(t)))$. This is a continuous path connecting $z_0$ and $z_1$, and $\gamma([0,1])\subset\Omega$. Furthermore, $\abs{\gamma(t)-\gamma_n(t)}<\varepsilon$ and so $\gamma([0,1])\subset B(0,r)$ as well.

Thus, $\Omega\cap B(0,r)$ is path-connected for all $r>0$. Similarly, $\Omega\setminus\overline{B(0,r)}$ is connected. This completes the proof that $\Omega$ is symmetric.
\end{proof}

We now show that among bounded simply connected symmetric domains, $h$-functions uniquely determine the domain.

\begin{theorem}\label{thm:unique} Suppose that $h_{\Omega}=h_{\widetilde\Omega}$, for two planar domains $\Omega$ and $\widetilde\Omega$ such that
\begin{thmenumerate}
\item \label{item:unique:bdd} $\Omega$ and $\widetilde\Omega$ are bounded and contain the point~$0$,
\item \label{item:unique:symm} $\Omega$ and $\widetilde\Omega$ are symmetric in the sense of \dfnref{symmetric}, and
\item \label{item:unique:sccd} $\Omega$ and $\widetilde\Omega$ are simply connected.
\end{thmenumerate}
Then $\Omega=\widetilde\Omega$.
\end{theorem}

\begin{proof}
Let $\mu=\sup\{r:h_\Omega(r)=0\}$, and let $M=\inf\{r:h_\Omega(r)=1\}$. By the properties of $h$-functions, $\mu$ and $M$ are the largest and smallest numbers, respectively, such that $B(0,\mu)\subset\Omega\subset B(0,M)$; since $h_\Omega = h_{\widetilde\Omega}$ we have that $B(0,\mu)\subset\widetilde\Omega\subset B(0,M)$.

Let $\Phi:\D\mapsto \Omega$, $\widetilde\Phi:\D\mapsto \widetilde\Omega$ be the Riemann maps; we may assume that $\Phi(0)=\widetilde\Phi(0)=0$, and that $\Phi'(0)>0$, $\widetilde\Phi'(0)>0$. Then by \rmkref{sccdsymm}, $\Phi((-1,1))=(-M,\mu)$.

Now, $\Phi$ is a bounded harmonic function defined on a $C^1$ domain. As is well-known (see for example \cite[Theorem~1.4.7]{K}), it follows that the radial limit $\lim_{r\to 1^-} \Phi(re^{i\theta})$ exists for almost every $\theta\in(-\pi,\pi]$. (In fact, a stronger notion of limit, called the \emph{non-tangential limit}, exists for a.e.~$\theta$.)

Suppose that for $k=0$, $1$, $\lim_{r\to 1^-} \Phi(re^{i\theta_k})$ exists for some $\theta_0$, $\theta_1$, with $0\leq\theta_0< \theta_1\leq\pi$. Consider the segments $I_k=\{te^{i\theta_k}:0\leq t<1\}$. Then $\Phi(I_k)$ is a path connecting $0$ to $\Phi(re^{ik})$. Furthermore, $\Phi(I_0)$ and $\Phi(I_1)$ do not intersect, and they lie entirely in the (closed) upper half-plane. Since Riemann maps are orientation-preserving, and $\lim_{r\to 1^-} \Phi(re^{i\theta_k})$ lies in $\partial\Omega$, by the symmetry condition we must have that $\left|\lim_{r\to 1^-} \Phi(re^{i\theta_1})\right|\geq\left|\lim_{r\to 1^-} \Phi(re^{i\theta_0})\right|$.

Thus, $\lim_{r\to 1^-} \abs{\Phi(re^{i\theta})}$ is defined for a.e.\ $\theta\in [0,\pi]$, and is nondecreasing.

If $\mu<\rho<M$, then let $E_\rho=\partial\Omega\cap\overline{B(0,\rho)}$, and let $u_\rho(z)=\omega(z,E_\rho,\Omega)$; recall from \bartoneqref{equation}{eqn:hmeasure} that $u_\rho$ is harmonic in~$\Omega$. Then let $v_\rho=u_\rho\circ\Phi$.

Suppose that $0\leq\theta\leq\pi$, and $\lim_{r\to 1^-}\abs{\Phi(re^{i\theta})}<\rho$. Then $\lim_{r\to 1^-}\abs{v_\rho(re^{i\phi})}=1$ for a.e.\ $\abs{\phi}\leq\theta$, and so $v_\rho(0)\geq \theta/\pi$. But $h_\Omega(\rho)=u_\rho(0)=v_\rho(0)$; thus if $\lim_{r\to 1^-}\abs{\Phi(re^{i\theta})}<\rho$ and $0\leq\theta\leq\pi$, then $h_\Omega(\rho)\geq \theta/\pi$.

Recall from \rmkref{sccdsymm} that $\Phi(\overline z)=\overline{\Phi(z)}$; therefore, if $-\pi\leq\theta\leq\pi$ and $\lim_{r\to 1^-}\abs{\Phi(re^{i\theta})}<\rho$, then $h_\Omega(\rho)\geq \abs{\theta}/\pi$.

Similarly, if $\lim_{r\to 1^-}\abs{\Phi(re^{i\theta})}>\rho$, then $h_\Omega(\rho)\leq \abs{\theta}/\pi$. Because $\Omega$ is simply connected, $h_\Omega$ is strictly increasing on $[\mu,M]$. Therefore, we may extend $h_\Omega^{-1}$ to a continuous function $[0,1]\mapsto [\mu,M]$. We then have that $\lim_{r\to 1^-} \abs{\Phi(re^{i\theta})}=h_\Omega^{-1}(\abs{\theta}/\pi)$ for a.e.~$\theta$.

Similarly, $\lim_{r\to 1^-} \abs{\widetilde\Phi(re^{i\theta})}=h_{\widetilde\Omega}^{-1}(\abs{\theta}/\pi)$. But since $h_{\widetilde\Omega}=h_\Omega$, this means that 
\[\lim_{r\to 1^-} \abs{\widetilde\Phi(re^{i\theta})/\Phi(re^{i\theta})}=1\]
for a.e.~$\theta$.

Consider $v(z)=\Phi(z)/\widetilde\Phi(z)$. Since $\Phi(0)=\widetilde\Phi(0)=0$ and $\Phi$, $\widetilde\Phi$ are injective on $\D$, $v$ is analytic with a removable singularity at $0$, and is never 0 on $\overline\D$.

Recall that $\re\log z$ is continuous on $\C\setminus\{0\}$ and that its value does not depend on the choice of branch cut of $\log$. Thus, $w(z)=\re \log (\Phi(z)/\widetilde\Phi(z))$ is continuous on~$\D$. For any given $z\in\D$, we may take the branch cut of $\log$ to avoid a neighborhood of $\Phi(z)/\widetilde\Phi(z)$. Thus $w(z)$ is harmonic in a neighborhood of any point $z\in\D$, and so $w(z)$ is harmonic on all of~$\D$.

Furthermore, $\lim_{r\to 1^-}w(rz)=0$ for a.e.\ $z\in \partial\D$; thus $\re\log v=w\equiv 0$ in $\D$, and so $\log v$ must be an imaginary constant on~$\D$. Since $\lim_{r\to 1^-}\Phi(r)=\mu=\lim_{r\to 1^-}\widetilde\Phi(r)$, we must have that $\log v\equiv 0$ and so $\Phi\equiv\widetilde\Phi$. Since $\Omega=\Phi(\D)$ and $\widetilde\Omega=\widetilde\Phi(\D)$, this implies that $\Omega=\widetilde\Omega$.\end{proof}

It is possible to  weaken slightly the conditions \eqref{item:unique:bdd} and \eqref{item:unique:sccd}.
Suppose that $\Omega$ is symmetric.

If $\Omega$ is simply connected, then $\Omega$ is bounded if and only if $\partial\Omega$ is bounded, which in turn is true if and only if $h_{\Omega}(M)=1$ for some finite number~$M$.

Conversely, suppose that $\Omega$ is symmetric, bounded and connected. We know that there exist numbers $\mu$, $M$ such that $h_{\Omega}=0$ on $[0,\mu)$, $0<h_{\Omega}<1$ on $(\mu,M)$, and $h_{\Omega}=1$ on $[M,\infty)$. Then $B(0,\mu)\subset\Omega$ up to a set of harmonic capacity zero. If $B(0,\mu)\subset\Omega$ and $\Omega\cap[\mu,M]$ is empty, then by symmetry $\Omega$ is simply connected. Again by symmetry, $\Omega\cap[\mu,M]$ is empty if and only if $h_{\Omega}$ is strictly increasing on $[\mu,M]$.

Thus, if $\widetilde\Omega$ is symmetric and either simply connected or bounded and connected, and if $h_{\widetilde\Omega}=h_\Omega$ for some domain $\Omega$ that satisfies all three conditions of \thmref{unique}, we have that up to a set of harmonic measure zero, $\widetilde\Omega$ must be bounded and simply connected and so the theorem holds.

However, we do need to require that $\widetilde\Omega$ be bounded or simply connected. In \cite[Example~1]{WalW01}, the authors showed that the bounded simply connected symmetric domain $B(-(M-\mu)/2, (M+\mu)/2)$ has the same harmonic measure distribution function as the unbounded, doubly connected symmetric domain $\C\setminus \overline{B((\mu+M)/2,(M-\mu)/2)}$.

As seen from the example of circle domains, we probably cannot weaken the condition \eqref{item:unique:symm} that both $\Omega$ and $\widetilde\Omega$ be symmetric. It is possible to weaken the requirement that $\Omega$ be simply connected by generalizing \lemref{circleunique}, but removing this requirement entirely  is beyond the scope of this paper.

We conclude this section by remarking that, if a function $f$ satisfies the conditions of \thmref{main} or \thmref{fjump}, then there exists a bounded simply connected domain $\Omega$ such that $f=h_\Omega$. By \lemref{mainsymmetric}, $\Omega$ is symmetric. So \thmref{unique} gives uniqueness, and the following theorem is proven.

\begin{theorem}
\label{thm:mainunique}
Let $f$ be a function that satisfies the conditions of \thmref{main} or \thmref{fjump}. Then there exists a domain $\Omega$ that is bounded and simply connected, has locally connected complement and is symmetric in the sense of \dfnref{symmetric}, such that $f=h_\Omega$. Furthermore, this domain $\Omega$ is unique up to sets of harmonic measure zero among bounded symmetric domains.
\end{theorem}

\section{Circle domains and estimates of harmonic measure}\label{sec:circle}

In this section we state and prove some lemmas, involving harmonic measure and circle domains, that we have used in this paper. We begin with a fundamental estimate of harmonic measure. Next come our estimates for individual circle domains, and finally our estimates for sequences of blocked circle domains.

Specifically, Lemmas~\ref{lem:channelstraight} and~\ref{lem:curvechannel} let us bound the harmonic measure of a set at the bottom of a channel (such as a short arc in a circle domain or a gate in a blocked circle domain). \lemref{fjump} puts a lower bound on the length of an arc in terms of its harmonic measure, while \lemref{deriv} controls the harmonic measure of an arc lying between two other arcs. 
\lemref{radial} puts a bound on the difference in harmonic measure distribution functions between a circle domain and the corresponding blocked circle domain (using Lemmas~\ref{lem:curvechannel} and~\ref{lem:altgatebound}). Finally,
\lemref{unifcctd} provides a sufficient condition for a sequence of blocked circle domains to be uniformly locally connected.

\subsection{A fundamental estimate on harmonic measure}

The following estimate
 is extremely useful.
\begin{lemma}[{\cite[Theorem~H.8]{GarM}}]\label{lem:channelstraight} Suppose $D$ is a finitely connected
Jordan domain and suppose \[F\subset
\left\{z\in\overline D :\re(z)\geq b\right\}.\]
Let $z_0\in D$ with $\re (z_0)= x_0<b$ and assume that for $x_0<x<b$,
$I_x\subset \{z\in D:\re(z)=x\}$ separates $z_0$ from $F$. If the length $\theta(x)$ of $I_x$ is measurable, then
\[\omega(z_0,F,D)\leq \frac{8}{\pi}\exp\left(
-\pi\int_{x_0}^b\frac{dx}{\theta(x)}\right).\]
\end{lemma}

\begin{figure}
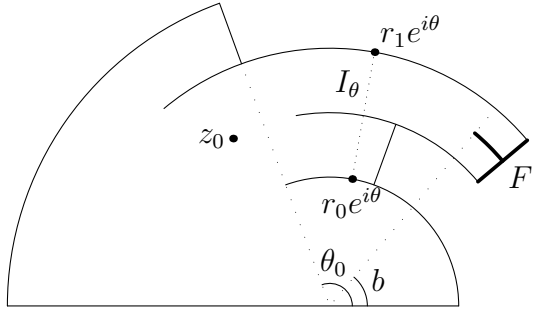

\mypic{4}
\caption{The curved domain of \lemref{curvechannel}: If the channel is long and thin, then the harmonic measure of $F$
from $z_0$ is small.}
\label{fig:channel}
\end{figure}

Roughly speaking, if a Brownian particle is released from a point $z_0$ inside a long, thin channel, the probability that the particle reaches the end of the channel before it hits the side increases with the width of the channel and decreases exponentially with the length of the channel.

We may adapt this lemma to blocked circle domains. See \figref{channel} for an illustration.
\begin{lemma}\label{lem:curvechannel} Suppose $D$ is a domain, and suppose
\[F\subset
\left\{z\in\overline D : 0 \leq \arg z \leq b, r_0 < \abs{z}< r_1\right\}.\]
Here we take $-\pi<\arg z\leq\pi$ for all $z\in\C$; we require $0\leq b$, $0<r_0<r_1$.

Assume that $D$ contains no points of the segment $[r_0,r_1]$ of the real line, and that there exists a $\theta_0>0$ such that  $r_k e^{i\theta}\notin D$ for $k=0$, $1$ and $0\leq\theta \leq \theta_0$.

Let $\widetilde D = \{re^{i\theta}:r_0<r<r_1, \> b<\theta<\theta_0\}$; we remark that $\widetilde D$ is a Jordan domain.
If $z_0 \in D\setminus \widetilde D$, then
\[\omega(z_0,F,D)\leq \frac{16}{\pi}\exp\left(-\pi r_0 \frac{\theta_0-b}{2(r_1-r_0)}\right).\]
\end{lemma}

\begin{proof} Our approach was suggested by the derivation of inequality~(19) in~\cite{SniW08}. Define $I_\theta:=\{re^{i\theta}:r_0<r<r_1\}$.

If $z_0\in D\setminus\widetilde D$ and $b<\theta<\theta_0$, then
\[\omega(z_0,F,D) \leq \sup_{z\in I_\theta} \omega(z,F,D).\]
Let $\theta_1-b = \theta_0-\theta_1$. If $z\in I_{\theta_1}$, then
\[\omega(z, F, D)\leq \omega(z, I_{\theta_0}\cup I_b, \widetilde D)
=2\omega(z,I_b,\widetilde D).\]

So
\[\omega(z_0,F,D) \leq 2 \sup_{z\in I_{\theta_1}} \omega(z,I_b,\widetilde D).\]

But $\widetilde D$ is a curved rectangle bounded away from zero. By applying the conformal map $\Phi(z)=i\log(z)$, we may transform $\widetilde D$ to the straight rectangle $\{z\in\C: -\theta_0 < \re z < -b,\> \log r_0 < \im z < \log r_1\}$. Then $\Phi(I_\theta)=\{-\theta+ir:\log r_0<r<\log r_1\}$.  \lemref{channelstraight} implies that
\[\sup_{z\in I_{\theta_1}} \omega(\Phi(z),\Phi(I_b),\Phi(\widetilde D))
\leq \frac{8}{\pi}\exp\left(-\pi \frac{\theta_0-b}{2\log(r_1/r_0)}\right).\]
Note that $\log (r_1/r_0)\leq (r_1-r_0)/r_0$.
Since $\omega(z,I_b,\widetilde D)=\omega(\Phi(z),\Phi(I_b),\Phi(\widetilde D))$, this completes the proof.\end{proof}

\subsection{Circle domains}

We seek to understand the harmonic measure distribution of circle domains more precisely. We begin by putting a lower bound on the arclength of an individual arc in terms of the harmonic measure of that arc.

\begin{lemma}\label{lem:fjump} Let $X$ be a circle domain of radius $M$, and let $A_{k}$ be the $k$th boundary arc of $X$, located at radius $r_k$ and with arclength $2\psi_{k}$.

Let $\beta=\omega(0,A_k,X)$. Suppose that $r_k\geq M(1-1/e)$. Then
\[\psi_k > \min\left\{\frac{\pi}{2},
\pi\beta -\frac{2}{\pi} \frac{M-r_k}{r_k}
\left(2\log\left(\frac{M}{M-r_k}\right)+\pi^2
\right)\right\}.
\]
\end{lemma}

We will use this lemma only when $(M-r_k)/M$ is small enough (depending on~$\beta$); if $(M-r_k)/M$ is large then the obvious inequality $\psi_k\geq 0$ provides a better estimate.

\begin{proof}
Write $A=A_k$, $\psi = \psi_k$, $r=r_k$. If $\beta=0$ then $\psi=0\geq\pi\beta$; if $\beta=1$ or $r=M$ then $\psi=\pi\geq \pi/2$, and so we are done. Otherwise, $0<\beta<1$ and $0<r< M$, and so $0<\psi<\pi$.

Let $u_A(z)=\omega(z,A,B(0,M)\setminus A)$; recall from \bartoneqref{Formula}{eqn:hmeasure} that $u_A$ is harmonic in $B(0,M)\setminus A$, $u_A=1$ on $A$ and $u_A=0$ on $\partial B(0,M)$. By the property of monotonicity in the domain of harmonic measure,
\begin{align*}
\beta&=\omega(0,A,X)\leq\omega(0,A,B(0,M)\setminus A)=u_A(0).
\end{align*}
Since $u_A$ is harmonic in $B(0,r)$, we have that
\begin{align*}
\beta&\leq u_A(0)
=\frac{1}{2\pi}\int_{-\pi}^\pi u_A(re^{i\theta})\,d\theta.
\end{align*}

We seek an upper bound on $u_A$ on $\partial B(0,r)$. Consider the function
\[u(z)=1-\frac{2}{\pi} \arg\left(\frac{M^\nu+iz^\nu}{M^\nu-i z^\nu}\right)\]
where $\nu>0$ is a positive real number (not necessarily an integer). We take the branch cut of $z^\nu$ and of $\arg$ to lie along the negative real axis.

Suppose that $\abs{z_0}\leq M$ and that $\abs{\arg z_0}\leq \pi/(2\nu)$. We claim that if $z_0\neq Me^{\pm i\pi/(2\nu)}$, then $u(z)$ is continuous and harmonic in a neighborhood of~$z_0$.

Since $\abs{\arg z_0}<\pi$, $M\pm iz^\nu$ is analytic in a neighborhood of~$z_0$. If $\abs{z_0}\leq M$ then $\re (M^\nu \pm iz_0^\nu)\geq M^\nu-\abs{z_0}^\nu\geq0$. If $\abs{z_0}<M$ or $\abs{\arg z_0}<\pi/(2\nu)$ then this inequality is strict, so $\re M^\nu \pm iz_0^\nu>0$. By continuity this inequality holds in a neighborhood of $z_0$. So the function $z\mapsto(M^\nu+iz^\nu)/(M^\nu-i z^\nu)$ is analytic and bounded away from the negative real axis (the branch cut of $\arg$) in a neighborhood of~$z_0$; thus $u(z)$ is harmonic in that neighborhood.

For all such $z$ we can write
\begin{align*}
u(z)&=1-\frac{2}{\pi} \arg\left(\frac{M^\nu+iz^\nu}{M^\nu-i z^\nu}\times\frac{M^\nu+i \overline{z^\nu}}{M^\nu+i \overline{z^\nu}}\right)
\\&=1-\frac{2}{\pi} \arg\left(
M^{2\nu} -\abs{z}^{2\nu} + 2i M^\nu \re (z^\nu)
\right).
\end{align*}

If $\abs{z}=M$ and $\abs{\arg z}<\pi/(2\nu)$, then
\[u(z)
= 1-\frac{2}{\pi}\arg\left(2i M^\nu \re (z^\nu)\right)
=0.\]
Similarly, if $|\arg z|= \pi/(2\nu)$ then $\re (z^\nu) = 0$, and so if in addition $\abs{z}<M$ then
\[u(z)=
1-\frac{2}{\pi} \arg\left(M^{2\nu} -\abs{z}^{2\nu}\right)=1
.\]

Choose $\nu=\pi/(2(\pi-\psi))$; then $\nu>1/2$ since $\psi>0$. Consider the domain
\[\Omega=\{z\in B(0,M): \abs{\arg z}>\psi\}\subset B(0,M)\setminus A.\]
Then $\partial\Omega$ consists of three smooth pieces: the two straight lines from the origin to $Me^{\pm i\psi}=-Me^{\pm i\pi/2\nu}$, and the arc (of radius $M$) connecting these two points and passing through the negative real axis.

Then $u_A$ is harmonic in $\Omega\subset B(0,M)\setminus A$ and continuous on~$\overline\Omega$, and the function $v(z)=u(-z)$ is harmonic in $\Omega$ and continuous on $\overline\Omega$ except at the points $Me^{\pm i\psi}$. We remark that the boundary values of $v$ are known; $v(z)=1$ on the straight segments and $v(z)=0$ on the boundary circle. $u_A(z)$ is also zero on the boundary circle; on the straight segments, $u_A(z)$ is unknown but satisfies $0<u_A(z)\leq 1=v(z)$ with equality holding only at $z=re^{\pm i\psi}$.

So by the maximum principle $0 < u_A(z)< v(z)< 1$ in~$\Omega$.
Recall that
\begin{align*}
\beta&\leq \frac{1}{2\pi}\int_{-\pi}^\pi u_A(re^{i\theta})\,d\theta
\end{align*}
where $r$ is the radius of the arc~$A$. Since $u_A(z) = 1$ on $A=\{re^{i\theta}:\abs{\theta}\leq \psi\}$, and $u_A(re^{i\theta}) = u_A(re^{-i\theta}) < u(re^{i(\pi-\theta)})$ for $\psi\leq \abs{\theta}\leq\pi$, we may rewrite this integral as
\begin{align*}
\frac{1}{2\pi}\int_{-\pi}^\pi u_A(re^{i\theta})\,d\theta
&<\frac{\psi}{\pi}
+\frac{1}{\pi}\int_{0}^{\pi-\psi} u(re^{i\theta})\,d\theta.
\end{align*}

But
\begin{align*}
\int_{0}^{\pi-\psi} u(re^{i\theta})\,d\theta
&=
\int_{0}^{\pi-\psi} 1-\frac{2}{\pi}\arg\left(
M^{2\nu} -r^{2\nu} + 2i M^\nu r^\nu\cos(\theta\nu)
\right)\,d\theta.
\end{align*}
Rewriting the $\arg$ function using arctangents and changing variables, we see that
\[\beta<\frac{\psi}{\pi}+\frac{\pi-\psi}{\pi}\int_{0}^{\pi/2}
\frac{4}{\pi^2} \arctan\frac{M^{2\nu}-r^{2\nu}}{2M^\nu r^\nu\cos\theta}\,d\theta.\]

Since $\arctan(x)\leq \min(x,\pi/2)$, we see that for any  $0\leq\theta_0<\pi/2$ we have
\begin{align*}
\int_{0}^{\pi/2}
\arctan\frac{M^{2\nu}-r^{2\nu}}{2M^{\nu}r^\nu\cos\theta}\,d\theta
&\leq
\int_{0}^{\theta_0}
\frac{M^{2\nu}-r^{2\nu}}{2M^{\nu}r^\nu\cos\theta}\,d\theta
+\int_{\theta_0}^{\pi/2} \frac{\pi}{2}\,d\theta
\\&\leq
\frac{M^{2\nu}-r^{2\nu}}{2M^{\nu}r^\nu}\log|\sec\theta_0+\tan\theta_0|+\frac{\pi}{2}\left(\frac{\pi}{2}-\theta_0\right)
\\&\leq
\frac{M^{2\nu}-r^{2\nu}}{2M^{\nu}r^\nu}\log|2\sec\theta_0|+\frac{\pi^2}{4}\cos\theta_0
\end{align*}
where the last inequality holds because $(\pi/2)-\theta_0\leq (\pi/2)\cos\theta_0$ for $0\leq\theta_0\leq\pi/2$.

Setting $x_0=\cos\theta_0$, we see that
\begin{align*}
\beta
%%%
&<
\frac{\psi}{\pi}+
\frac{4}{\pi^2} \frac{\pi-\psi}{\pi}
\left(\frac{M^{2\nu}-r^{2\nu}}{2M^{\nu}r^\nu}\log\frac{2}{x_0}+\frac{\pi^2}{4}x_0\right)
\end{align*}
for any $x_0$ with $0< x_0\leq 1$, where $\nu=\nu(\psi)=\pi/(2(\pi-\psi))$. This inequality holds for any arc $A$, of any radius or arclength, with $0<\beta<1$. Unfortunately, we seek a lower bound on $\psi$ and not an upper bound on~$\beta$, and so we must choose $x_0$ and then solve for $\psi$; to do so, we assume $r\geq M(1-1/e)$ and $\psi\leq \pi/2$ (implying $\nu\leq 1$ and that $A$ lies in the right half-plane).

Suppose $r/M\geq 1-1/e$. Then since $\nu>1/2$ we have
\[2(1-(r/M)^\nu)\leq 2(1-\sqrt{r/M})\leq 2(1-\sqrt{1-1/e})<1,\]
and so we may choose $x_0=2(1-(r/M)^\nu)$. Now,
\begin{align*}
\beta
%%%
&<
\frac{\psi}{\pi}+
\frac{4}{\pi^2} \frac{\pi-\psi}{\pi}
\left(\frac{M^{2\nu}-r^{2\nu}}{2M^{\nu}r^\nu}\log\frac{1}{1-(r/M)^\nu}+\frac{\pi^2}{2}(1-(r/M)^\nu)\right)
%%%
\\&=
\frac{\psi}{\pi}+
\frac{2}{\pi^2} \frac{\pi-\psi}{\pi}
\left(\frac{M^{\nu}+r^{\nu}}{r^\nu}\frac{M^\nu-r^{\nu}}{M^\nu}\log\frac{M^\nu}{M^\nu-r^\nu}+\pi^2\frac{M^\nu-r^\nu}{M^\nu}\right).
\end{align*}
Since $\nu\leq 1$ we have $(M^\nu-r^\nu)/M^\nu \leq (M-r)/M$; by our assumption on $r$ we have $(M-r)/M\leq 1/e$. The function $x\mapsto x\log (1/x)$ is increasing on $(0,1/e)$. So
\begin{align*}
\beta
%%%
&<
\frac{\psi}{\pi}+
\frac{2}{\pi^2} \frac{\pi-\psi}{\pi}
\left(\frac{M+r}{r}\frac{M-r}{M}\log\frac{M}{M-r}+\pi^2\frac{M-r}{M}\right)
%%%
\\&<
\frac{\psi}{\pi}+
\frac{2}{\pi^2}\frac{M-r}{r}
\left(2\log\frac{M}{M-r}+\pi^2\right)
.\end{align*}
Solving for $\psi$ completes the proof.
\end{proof}

Next, we prove a lemma that controls the harmonic measure of arcs lying in between other arcs. This lemma is somewhat technical, but very useful.
Informally, \lemref{deriv} states that, if the $h$-function of a circle domain $X_n$ approximates some function~$f$, and the minimal secant slope of $f$ is positive, then any boundary arc long enough to be almost a full circle must be close to the outer boundary circle, and if two boundary arcs are close together then the arcs between them cannot be too short.

These conditions mirror the conditions of \lemref{unifcctd}, which provides a sufficient condition for a sequence of blocked circle domains to be uniformly locally connected. Thus, \lemref{deriv} and \lemref{unifcctd} are vital to the proof of \thmref{main}. Also, \bartoneqref{condition}{item:deriv:a} in combination with \lemref{fjump} provides sufficient conditions to ensure that none of the boundary arcs are very short, and so Lemmas~\ref{lem:fjump} and~\ref{lem:deriv} combine to give \thmref{fjump}.

\begin{lemma}\label{lem:deriv} Let $f$ be a function that is a candidate for a harmonic measure distribution function, as in \secref{sufficient}, and define the numbers $\mu$ and $M$ and the circle domains $X_n$ as in that section. Let $\alpha$ be the minimal secant slope of~$f$, as in \bartoneqref{formula}{eqn:derivmin}. Fix some $n>0$. For each $0\leq j\leq n$ and each $0\leq k\leq n$, write $X=X_n$, $A_k=A_{n,k}$, $r_k=r_{n,k}$, $\psi_{k}=\psi_{n,k}$, $\eta_{j,k}=\eta_{n,j,k}$.

Suppose that $\alpha>0$. Then the following estimates hold.
\begin{thmenumerate}
\item \label{item:deriv:b} If $0\leq k\leq n$, then
\[M-r_{k}\leq\frac{M-\mu}{\alpha\pi}(\pi-\psi_k).\]
\item\label{item:deriv:a} If $0\leq j<k\leq n$, then the depth $\eta_{j,k}$ of the shortest arc between $A_j$ and $A_k$ satisfies
\[\eta_{j,k}\leq  \frac{4}{\pi\mu} (r_k-r_j)
\left(
\log\left(\frac{M-\mu}{\alpha(r_k-r_j)} \right)
+
\log\left(\frac{256}{\pi}\right)
\right)
.\]
\end{thmenumerate}
Therefore, for every $\varepsilon>0$, there exist numbers $\delta_1>0$, $\delta_2>0$ depending only on $\alpha$, $\mu$, $M$ and $\varepsilon$ (in particular, not on~$n$) such that
\begin{thmenumerate}
\item\label{item:deriv:d} If\/ $0\leq k\leq n$ and $\pi-\psi_{k}<\delta_2$, then $M-r_{k}<\varepsilon$, and
\item\label{item:deriv:c} If\/ $0\leq j<k\leq n$ and ${r_{k}-r_{j}}<\delta_1$, then $\eta_{j,k}<\varepsilon$.
\end{thmenumerate}
\end{lemma}

The conclusion \eqref{item:deriv:c} follows from \eqref{item:deriv:a} because the function
\begin{equation}\label{eqn:chiinfty}
\chi_\infty(\delta)=\frac{4}{\pi\mu} \delta
\left(
\log\left(\frac{M-\mu}{\alpha\delta} \right)
+
\log\left(\frac{256}{\pi}\right)
\right)
\end{equation}
is increasing on $(0,M-\mu]$ and satisfies $\lim_{\delta \to 0^+} \chi_\infty(\delta)=0$.
We remark that by the definition \eqref{dfn:eta} of~$\eta_{j,k}$, if $A_j$ and $A_k$ are two boundary arcs, then
$\eta_{j,k}\leq\min(\psi_j,\psi_k).$ Thus, if $\psi_j$ or $\psi_k$ is small then \bartoneqref{condition}{item:deriv:a} holds automatically; \bartoneqref{condition}{item:deriv:a} is of interest mainly if both of the arcs $A_j$ and $A_k$ are relatively long.

\begin{proof} We consider \bartoneqref{condition}{item:deriv:b}  first. Simply note that by definition of~$\alpha$,
\begin{align*}
M-r_{k}&
\leq \frac{M-\mu}{\alpha} (f(M)-f(r_{k}))
= \frac{M-\mu}{\alpha}(h_X(M)-h_X(r_{k}))
\\&= \frac{M-\mu}{\alpha}\omega(0, \partial X\setminus \overline{B(0,r_{k})},X)
\leq\frac{M-\mu}{\alpha}\omega(0,\partial B(0,r_{k})\setminus\partial X, B(0,r_{k}))
\\&= (\pi-\psi_{k})\frac{M-\mu}{\alpha\pi}.
\end{align*}
Thus if $\psi_{k}$ is near $\pi$, then $M-r_{k}$ is near~0.

\bartoneqref{Condition}{item:deriv:a} is much more complicated.

Recall that the $j$th boundary arc of $X$ is denoted $A_j$, is a subset of the circle of radius $r_j$ centered at zero, and has arclength $2\psi_j$. If $A_j$ and $A_k$ are two boundary arcs with $j<k$, then $r_j<r_k$, and if another arc $A_l$ lies between $A_j$ and $A_k$, then $r_j<r_l<r_k$ and so $j<l<k$.

\textbf{Claim:} Let $p$ be an integer with $p\geq 1$. Then there is an increasing function $\chi_p(\delta)$, defined for $0\leq \delta\leq M-\mu$, that satisfies the following condition. Suppose that $A_j$ and $A_k$ are two boundary arcs with $r_j<r_k$, and that there are $m$ arcs $A_l$ between $A_j$ and~$A_k$. Then there are at most $m/3^p$ such arcs $A_l$ that satisfy
\[\psi_l<\min(\psi_j,\psi_k)-\chi_p(r_k-r_j).\]
The function $\chi_p$ depends only on $p$, $\alpha$, $\mu$ and~$M$. In particular, $\chi_p(\delta)$ does not depend on~$n$, the number of boundary arcs.

Suppose that the claim holds. If $A_j$ and $A_k$ are any pair of arcs with $j<k$, let $m$ be the number of arcs $A_l$ that lie between $A_j$ and $A_k$. The number of such arcs $A_l$ which also satisfy $\psi_l<\min(\psi_j,\psi_k)-\chi_p(r_k-r_j)$ is an integer which is at most $m/3^p$. Therefore, if $3^p>m$, then none of the arcs $A_l$ between $A_j$ and $A_k$ satisfy $\psi_l<\min(\psi_j,\psi_k)-\chi_p(r_k-r_j)$. That is, none of these arcs have arclength less than $2(\min(\psi_j,\psi_k)-\chi_p(r_k-r_j))$.

Recalling the \bartoneqref{definition}{dfn:eta} of $\eta_{j,k}$ as the depth of the shortest arc between $A_j$ and~$A_k$, we see that $\eta_{j,k}\leq \sup_{p\geq 1} \chi_p(r_k-r_j)$. We will complete the proof of \lemref{deriv} by showing that we may choose the functions $\chi_p(\delta)$ such that $\sup_{p\geq 1} \chi_p(\delta)=\chi_\infty(\delta)$, where $\chi_\infty(\delta)$ is given by \bartoneqref{formula}{eqn:chiinfty} and $\chi_\infty(r_k-r_j)$ equals the quantity on the right-hand side of \bartoneqref{formula}{item:deriv:a}.

We first find $\chi_1(\delta)$, that is, establish the claim for $p=1$; we will use induction to bound $\chi_p(\delta)$ for $p>1$.

Let $A_j$ and $A_k$ be any two boundary arcs of $X$ with $r_j<r_k$. Let $\psi=\min(\psi_j,\psi_k)$. We remark that if $\psi=0$ then the claim holds for any nonnegative functions $\chi_p(\delta)$; we will therefore assume $\psi>0$.

Let $m=k-j-1$ be the number of arcs between $A_j$ and~$A_k$. If $m=0$ then the claim holds for any functions $\chi_p(\delta)$; we therefore assume $m\geq 1$.

Fix some number $\chi$ with $0<\chi<\psi$. If $A_l$ is between $A_j$ and $A_k$, and $\psi_l<\psi-\chi$, we call $A_l$ a \emph{$\chi$-short arc}.
Let
\[\B(\chi)=\left\{A_l:\psi_l<\psi-\chi,\>j<l<k\right\}\]
be the set of $\chi$-short arcs between $A_j$ and~$A_k$. We want to find a function $\chi_1(\delta)$ such that $\abs{\B(\chi_1(r_k-r_j))}\leq\frac{1}{3}m$. We will do this by bounding the harmonic measure of these arcs.

\begin{figure}
\mypic{10}
\caption{The boundary arcs and auxiliary radial line segments $L_\pm$ for \bartoneqref{condition}{item:deriv:a} of \lemref{deriv}. The set $B_+\cup L_+$ is shown in bold.}\label{fig:deriv}
\end{figure}

Let $0<\lambda<\psi-\chi$, and let $L_+$ and $L_-$ be radial line segments of inner radius $r_j$, outer radius $r_k$, at angle $\pm \lambda$. See \figref{deriv}. Let $X'$ be the connected component of $X\setminus(L_+\cup L_-)$ containing~0.

Let $B=\cup_{\psi_l<\psi-\chi} A_l$ be the union of the $\chi$-short arcs, and let $B_+=\{z\in B:\arg z\geq\lambda\}$. The set $B_+\cup L_+$ is drawn in bold in \figref{deriv}.
Then
\[\omega(0,B,X) < \omega(0,B\cup L_+\cup L_-,X')
= 2\omega(0,B_+\cup L_+, X').
\]
By \lemref{curvechannel},
\[
\omega(0,B_+\cup L_+,X')
\leq \frac{16}{\pi}\exp\left(-\frac{\pi \mu \chi}{2(r_k-r_j)}\right).
\]
Notice that by definition of $\alpha$ and~$X$, $\omega(0,A_l,X)=f(r_l)-f(r_{l-1})\geq \alpha/n$ provided $l>0$. Therefore,  $\omega(0,B,X)\geq \alpha\abs{\B(\chi)}/n$. So
\begin{equation}
\label{eqn:Bchi}
\frac{\alpha}{n}\abs{\B(\chi)}
\leq \frac{32}{\pi}\exp\left(-\frac{\pi \mu \chi}{2(r_k-r_j)}\right).
\end{equation}

We wish to control the right-hand side. Let $\chi_1(0)=0$. If $0<\delta\leq M-\mu$, define $\chi_1(\delta)$ to be the number that satisfies
\[
\frac{32}{\pi}\exp\left(-\frac{\pi \mu \chi_1(\delta)}{2\delta}\right)
=
\frac{\alpha}{4}\frac{\delta}{M-\mu}.
\]
We may solve for $\chi_1(\delta)$ to see that
\begin{equation}
\label{eqn:chidelta}
\chi_1(\delta)
=\frac{2}{\pi \mu}\delta \log\left(\frac{128}{\pi\alpha}\frac{M-\mu}{\delta}\right).
\end{equation}
Observe that $\lim_{\delta\to 0^+}\chi_1(\delta)=0$, and so the function $\chi_1$ is continuous on $[0,M-\mu]$. Furthermore, if $C>0$ then the function $\delta\mapsto \delta\log (C/\delta)$ is increasing on $(0,C/e]$. Since $0<\alpha<1$, we have that $128/(\pi\alpha)>e$, and so $\chi_1(\delta)$ is increasing for $\delta\in [0,M-\mu]$.

If $\psi\leq\chi_1(r_k-r_j)$, then $\psi-\chi_1(r_k-r_j)\leq 0$, and so none of the arcs $A_l$ between $A_j$ and $A_k$ satisfy $\psi_l<\psi-\chi_1(r_k-r_j)$. In this case, the claim holds for $p=1$. Otherwise, $0<\chi_1(r_k-r_j)<\psi$ and we may apply \bartoneqref{formula}{eqn:Bchi} with $\chi=\chi_1(r_k-r_j)$.

Recall that $r_k= \mu+k(M-\mu)/n$. Thus, there are $n(r_k-r_j)/(M-\mu)-1$ boundary arcs lying between $A_j$ and~$A_k$; therefore, $m=n(r_k-r_j)/(M-\mu)-1$.

It follows that
\[
\abs{\B(\chi_1(r_k-r_j))}
\leq\frac{1}{4}\frac{n}{M-\mu}({r_k-r_j})
=\frac{m+1}{4}
.\]
$\abs{\B(\chi_1(r_k-r_j))}$ and $m$ are nonnegative integers. Thus, this inequality implies that 
$\abs{\B(\chi_1(r_k-r_j))}\leq m/3$.

Thus, the claim holds for $p=1$, with $\chi_1(\delta)$ given by \bartoneqref{equation}{eqn:chidelta}.

For the inductive step, suppose that the function $\chi_p(\delta)$ exists and is finite. We want to show that $\chi_{p+1}(\delta)$ exists. Pick two boundary arcs $A_j$ and $A_k$ with $r_j<r_k$. There are $m=k-j-1$ arcs between $A_j$ and $A_k$, not including $A_j$ and~$A_k$. Define $\psi=\min(\psi_j,\psi_k)$. Recall that if $A_l$ lies between $A_j$ and $A_k$, and $\psi_l<\psi-\chi$, then we call $A_l$ a \emph{$\chi$-short arc}.

We wish to find a $\chi_{p+1}(\delta)$ such that at most $3^{-p-1}m$ arcs are $\chi_{p+1}(r_k-r_j)$-short arcs. By definition of $\chi_p(\delta)$, at most $3^{-p}m$ arcs are $\chi_p(r_k-r_j)$-short arcs. We refer to $\chi_p(r_k-r_j)$-short arcs simply as \emph{short arcs}. If we choose the function $\chi_{p+1}$ such that $\chi_{p+1}(\delta)\geq \chi_p(\delta)$ for all~$\delta$, then all $\chi_{p+1}(r_k-r_j)$-short arcs are short arcs. It is this assumption that allows us to work by induction.

If none of the arcs $A_l$ between $A_j$ and $A_k$ are short arcs, then the claim holds for any $\chi_{p+1}(\delta)\geq \chi_p(\delta)$. In particular, if $\psi\leq \chi_p(r_k-r_j)$ then the claim holds. Therefore, we assume that at least one arc $A_l$ is a short arc.

We remark that the arcs $A_j$ and $A_k$ are not short. It is possible to find numbers $j_s$ and $k_s$, for $1\leq s\leq S$, such that the following conditions hold.
\begin{itemize}
\item $A_{j_s}$ and $A_{k_s}$ are not short arcs.
\item If $A_l$ lies between $A_{j_s}$ and $A_{k_s}$, then $A_l$ is a short arc.
\item Conversely, if $A_l$ is a short arc between $A_j$ and $A_k$, then there is exactly one number $s$ with $1\leq s\leq S$ such that $A_l$ lies between $A_{j_s}$ and $A_{k_s}$.
\end{itemize}
We may further require that if $1\leq s\leq S$, then $j\leq j_s<k_s\leq k$ and there is at least one (necessarily short) arc $A_l$ between $A_{j_s}$ and $A_{k_s}$.

We begin our analysis of the short arcs by bounding $r_{k_s}-r_{j_s}$ for each~$s$.
There are $k_s-j_s-1\geq 1$ arcs lying between $A_{j_s}$ and $A_{k_s}$, and they are all short arcs. There are at most $3^{-p}m=3^{-p}(k-j-1)$ short arcs between $A_j$ and $A_k$. Therefore,
\[1\leq k_s-j_s-1\leq 3^{-p}m
=3^{-p}(k-j-1).\]
Recall that $r_l=\mu+l(M-\mu)/n$. Therefore,
\[r_{k_s}-r_{j_s}
= (k_s-j_s)\frac{M-\mu}{n}
= (k_s-j_s)\frac{r_k-r_j}{k-j}
\leq (r_k-r_j)\frac{k_s-j_s}{1+3^p(k_s-j_s-1)}.
\]
Since $k_s-j_s\geq 2$, and $p\geq 1$, we have that
\[\frac{k_s-j_s}{1+3^p(k_s-j_s-1)}
\geq \frac{2}{3^p+1}
\geq 2^{-p}\]
where the second inequality is chosen for the sake of simplicity.
Thus, if $1\leq s\leq S$ then $r_{k_s}-r_{j_s}\leq 2^{-p}(r_k-r_j)$.

We now apply this bound. Let $m_s=k_s-j_s-1$ denote the number of arcs between $A_{j_s}$ and $A_{k_s}$.
By definition of $\chi_1(\delta)$, there are at most $m_s/3$ boundary arcs $A_l$ between $A_{j_s}$ and $A_{k_s}$ that satisfy \[\psi_l<\min(\psi_{j_s},\psi_{k_s})-\chi_1(r_{k_s}-r_{j_s}).\]
Recall that $\chi_1(\delta)$ is an increasing function. Furthermore, since $A_{j_s}$ and $A_{k_s}$ are not short, we have that $\min(\psi_{j_s},\psi_{k_s})\geq \psi-\chi_p(r_k-r_j)$. Therefore, if $1\leq s\leq S$ then
\begin{equation}\label{eqn:psi2}
\min(\psi_{j_s},\psi_{k_s})-\chi_1(r_{k_s}-r_{j_s})
\geq \psi-\chi_p(r_k-r_j)-\chi_1(2^{-p}(r_k-r_j)).
\end{equation}

Suppose that $A_l$ lies between $A_j$ and $A_k$ and satisfies
\begin{equation}\label{eqn:psi1}
\psi_l \leq \psi-\chi_p(r_k-r_j)-\chi_1(2^{-p}(r_k-r_j)).
\end{equation}
Then $A_l$ is a short arc, and so there is some number $s$ such that
$j_s<l<k_s$. By \bartoneqref{equation}{eqn:psi2}, $\psi_l<\min(\psi_{j_s},\psi_{k_s})-\chi_1(r_{k_s}-r_{j_s})$.
For each $s$ there are at most $m_s/3$ such arcs. Therefore, there are at most $\sum_{s=1}^S m_s/3$ arcs $A_l$ between $A_j$ and $A_k$ that satisfy \bartoneqref{equation}{eqn:psi1}. But $\sum_{s=1}^S m_s$ is equal to the number of short arcs between $A_j$ and $A_k$, which by definition is at most $3^{-p}m$.

Therefore, there are at most $3^{-p-1}m$ arcs $A_l$ between $A_j$ and $A_k$ that satisfy
\[\psi_l \leq \psi-\chi_p(r_k-r_j)-\chi_1(2^{-p}(r_k-r_j)).\]
Thus, the claim is established, with
\[\chi_{p+1}(\delta)=\chi_p(\delta)+\chi_1(2^{-p}\delta)
=\sum_{q=0}^{p} \chi_1(2^{-q}\delta).\]

Observe that since $\chi_1(2^{-q}\delta)>0$ for all $q>0$ and all $0<\delta\leq M-\mu$, we have that
\begin{align*}
\sup_{p\geq 1}\chi_p(\delta)
&=
\sum_{q=0}^\infty \chi_1\left(2^{-q}\delta\right)
%\\&
=
\frac{2}{\pi \mu} \delta
\sum_{q=0}^\infty
\frac{1}{2^q}
\left(
\log\left(\frac{M-\mu}{\alpha\delta} \right)
+\log\left(\frac{128}{\pi}\right)
+q\log 2
\right)
\\&=
\frac{4}{\pi \mu} \delta
\left(
\log\left(\frac{M-\mu}{\alpha\delta} \right)
+
\log\left(\frac{128}{\pi}\right)+\log 2
\right).
\end{align*}
This expression is precisely the $\chi_\infty(\delta)$ of \bartoneqref{formula}{eqn:chiinfty}.

Finally, recall that $\eta_{j,k} \leq \sup_{p\geq 1}\chi_p(r_k-r_j)$. So
\begin{align*}
\eta_{j,k}
&\leq \sup_{p\geq 1}\chi_p(r_k-r_j)
\leq
\frac{4}{\pi \mu} (r_k-r_j)
\left(
\log\left(\frac{M-\mu}{\alpha(r_k-r_j)} \right)
+
\log\left(\frac{256}{\pi}\right)
\right)
\end{align*}
as desired.
\end{proof}

\subsection{Blocked circle domains}
\label{sec:hXtohOmega}

In our proof of \thmref{main}, we needed conditions guaranteeing that $h_{X_n}-h_{\Omega_n}\to 0$. This means that we want estimates on the harmonic measure of the gates of blocked circle domains. In this section we establish those conditions. We follow the development originally given in \cite{SniW08}, but adapted to our situation.

Throughout this section, let $X$ be a circle domain and $\Omega$ a blocked circle domain with $\Omega\subset X$, as in \dfnref{circle}. Let their boundary arcs be at radii $r_k$ and have arclength $2\psi_k$, and let the gates be at angles $\pm\phi_k$. (See \figref{param}.) Let $\chi_k=\min(\psi_k,\psi_{k+1})-\phi_k$. That is, $\chi_k$ measures the depth of the gate in its channel.

\lemref{curvechannel} will let us control the harmonic measure of most of the gates. In \lemref{altgatebound}, we develop an alternate bound for gates that lie along the positive real axis.
\begin{lemma}\label{lem:altgatebound} If $l$ is a gate connecting the arcs at radii $r_{k}$ and $r_{k+1}$, and if $\phi_k=0$, so that $l$ lies along the positive real axis, then
\[\omega(0,l,\Omega)\leq
\frac{2}{\pi}
\sqrt{\frac{r_{k+1}-r_k}{r_{k}}}.
\]
\end{lemma}

\begin{proof} Define $\Psi:=B(0,M)\setminus[r_{k},M]$ to be a disk minus a slit. Then since $\Omega\subset\Psi$,
\[\omega(0,l,\Omega)\leq \omega(0,l,\Psi)=\omega(0,[r_k,r_{k+1}],\Psi).\]
The harmonic measure on the right-hand side may be computed explicitly: we transform $\Psi$ to the upper half-plane $\U$ via the conformal map
\[z\mapsto \sqrt{\left(\frac{1}{r_{k}+M}-\frac{1}{2M}\right)^2-\left(\frac{1}{z+M}-\frac{1}{2M}\right)^2}.\]

The point $0$ is mapped to a point $it$ on the positive imaginary axis, and $[r_{k},r_{k+1}]$ to an interval $[-r',r']$ in $\R=\partial\U$.
Here
\begin{align*}
t&=\sqrt{\left(\frac{1}{2M}\right)^2
-\left(\frac{1}{r_{k}+M}-\frac{1}{2M}\right)^2}>0,\quad\text{and}\\
r'& = \sqrt{\left(\frac{1}{r_{k}+M}-\frac{1}{2M}\right)^2-\left(\frac{1}{r_{k+1}+M}-\frac{1}{2M}\right)^2}>0.
\end{align*}
Using the harmonic function $\frac{1}{\pi}\arg(z-r')-\frac{1}{\pi}\arg(z+r')$ we may compute
\begin{align*}
\omega(0,[r_{k},r_{k+1}],\Psi)
&=\omega(it,[-r',r'],\U)
=\frac{2}{\pi}\arctan\frac{r'}{t}
\\&=\frac{2}{\pi}
\arctan
\sqrt{\frac{(r_{k+1}-r_{k})(M^2-r_{k+1}r_{k})}{{r_{k}(r_{k+1}+M)^2}}}
\leq\frac{2}{\pi}
\sqrt{\frac{r_{k+1}-r_k}{r_k}}.
\end{align*}
We omit the details.
\end{proof}

Recall that, in \thmref{main}, we needed $h_{\Omega_n}-h_{X_n}\to 0$. That is, we needed the harmonic measure distribution functions of blocked circle domains $\Omega_n$ to approach those of the corresponding circle domains~$X_n$. In the next lemma, we assemble the known results involving gates in order to achieve a uniform bound on $h_X-h_\Omega$.

\begin{lemma} \label{lem:radial} Let $X$ be a circle domain, and let $\Omega\subset X$ be a blocked circle domain such that $\partial\Omega\setminus\partial X$ is a union of gates. Define $r_k$, $\psi_k$, and $\phi_k$, $\chi_k$ as in \secref{sufficient}, so that $r_k$ is the radius of the $k$th arc, $\chi_k$ is the inset angle of the $k$th pair of gates, and so on.

Then for all $r\in[\mu,M]$,
\[\abs{h_{X}(r)-h_{\Omega}(r)}
\leq \sum_{k:\phi_k>0}\frac{32}{\pi}
\exp\left(-\frac{\pi \, r_{k} \, \chi_{k}}{2(r_{k+1}-r_{k})}\right)
+\sum_{k:\phi_k=0} \frac{2}{\pi}\sqrt{\frac{r_{k+1}-r_k}{r_k}}
.\]
\end{lemma}

\begin{proof} Fix $r\in[\mu,M]$. We first show that $\abs{h_X(r)-h_\Omega(r)}$ is bounded by the harmonic measure of the gates in~$\Omega$.
Recall that
\begin{align*}
h_{X}(r) &:= \omega(0, \overline{B(0,r)}\cap\partial X, X),
\\
h_{\Omega}(r) &:= \omega(0, \overline{B(0,r)}\cap\partial \Omega, \Omega)
.\end{align*}
Let $E_r=\overline{B(0,r)}\cap\partial X$ and $F_r=\overline{B(0,r)}\cap\partial\Omega$ be the portions of $\partial X$ and $\partial\Omega$, respectively, lying within a distance $r$ of zero. Then  $h_X(r)=\omega(0,E_r,X)$ and $h_\Omega(r)=\omega(0,F_r,\Omega)$, and so we need only bound $\abs{\omega(0,E_r,X)-\omega(0,F_r,\Omega)}$.

We may transform $X$ to $\Omega$ by adding the gates (and deleting disconnected components). Let $G=\partial\Omega\setminus\partial X$ be the union of the gates, that is, the newly-added boundary.
Then $F_r = (E_r\cap\partial\Omega) \cup (G\cap \overline{B(0,r)})$, and $E_r\cap G$ is empty.

Then by monotonicity in the domain, and since $\partial\Omega\setminus\partial X = G$,
\[\omega(0,E_r,X)\leq \omega(0, G,\Omega) + \omega(0, E_r\cap\partial\Omega, \Omega) \leq \omega(0,G,\Omega) + \omega(0,F_r,\Omega)\]
and
\[\omega(0,F_r,\Omega)\leq \omega(0,G,\Omega)+\omega(0,E_r\cap\partial\Omega,\Omega)
\leq \omega(0,G,\Omega)+\omega(0,E_r,X).\]

So $\abs{h_X(r)-h_\Omega(r)}=\abs{\omega(0,E_r,X)-\omega(0,F_r,\Omega)}\leq \omega(0,G,\Omega)$.

Next, we estimate the harmonic measure of the gates.
Write $G=G_0\cup G_1$, where $G_0$ is the union of gates that lie along the real axis and $G_1$ is the union of gates that do not. By \lemref{altgatebound},
\[\omega(0,G_0,\Omega)\leq\sum_{k:\phi_k=0} \frac{2}{\pi}\sqrt{\frac{r_{k+1}-r_k}{r_k}}.\]

Let $l\subset G_1$ be a gate that lies between the arcs $A_k$ and~$A_{k+1}$. We apply \lemref{curvechannel} with $D=\Omega$, $z_0=0$, $F=l$, $b=\phi_k$, $\theta_0=\min(\psi_k,\psi_{k+1})$, and $r_0=r_k$, $r_1=r_{k+1}$. By definition of $\psi_k$ and $\phi_k$, the conditions of \lemref{curvechannel} hold; thus,
\[\omega(0,l,\Omega)
\leq
\frac{16}{\pi}
\exp\left(-\frac{\pi r_{k} \chi_{k}}{2(r_{k+1}-r_{k})}\right)
.\]

Combining these estimates, and noting that if $\phi_k>0$ then there are two gates between the arcs $A_k$ and $A_{k+1}$, we have that
\begin{align*}
\omega(0,G,\Omega)
&\leq
\sum_{k:\phi_k>0}\frac{32}{\pi}
\exp\left(-\frac{\pi r_{k} \chi_{k}}{2(r_{k+1}-r_{k})}\right)
+\sum_{k:\phi_k=0} \frac{2}{\pi}\sqrt{\frac{r_{k+1}-r_k}{r_k}}
.\qedhere\end{align*}
\end{proof}

\subsection{Sequences of blocked circle domains}

Finally, we move to sequences $\{\Omega_n\}_{n=1}^\infty$ of blocked circle domains. In the proof of \thmref{main}, we used a simple condition for such a sequence to have uniformly locally connected complements. We now prove that implication.

\begin{lemma}\label{lem:unifcctd} Let $X_n$, $\Omega_n$, $r_{n,k}$, $\psi_{n,k}$, $\phi_{n,k}$, $\theta_{n,j,k}$ be defined as in \secref{sufficient}. Then $\{\Omega_n\}_{n=1}^\infty$ is a sequence of blocked circle domains.

Suppose that for each $\varepsilon>0$ there exist positive numbers $\delta_1(\varepsilon)$, $\delta_2(\varepsilon)$ such that the following conditions hold.
\begin{thmenumerate}
\item \label{item:unifcctd:b} If $0\leq k\leq n$ and  $\pi-\psi_{n,k}<\delta_2(\varepsilon)$, then $M-r_{n,k}<\varepsilon$, and
\item \label{item:unifcctd:a} If $0\leq j<k\leq n$ and ${r_{n,k}-r_{n,j}}<\delta_1(\varepsilon)$ then $\theta_{n,j,k}<\varepsilon$.
\end{thmenumerate}
Then $\{\C\setminus\Omega_n\}_{n=1}^\infty$ is uniformly locally connected.
\end{lemma}

\begin{proof}
Fix some $\varepsilon>0$. We wish to find some $\delta>0$ such that if $z$, $w\notin\Omega_n$ and $\abs{z-w}<\delta$, then there is a continuum of diameter at most $\varepsilon$ contained in $\C\setminus\Omega_n$ connecting $z$ and $w$. We need only consider the case where $z$, $w\in\partial\Omega_n$. In fact, we can go further and consider only the case where $z$, $w$ lie on boundary arcs (not gates), that is, where $z$, $w\in\partial X_n$.

Define
\[\delta:=\min\left\{
\mu,
\frac{\varepsilon}{C},
\delta_1(\varepsilon/MC),
\frac{\mu}{\pi}\delta_2(\varepsilon/C),
\frac{\mu}{\pi}\delta_2(\delta_1(\varepsilon/MC))
\right\},\]
where $C$ is a constant to be chosen later.

We will use two elementary geometric estimates throughout this proof. First, let $S=\{z: r<\abs{z}<R, \theta_0<\arg z<\theta_0+\theta\}$ be an annular sector with inner radius $r$, outer radius $R$, and subtending an angle of $\theta$. Then its diameter is bounded by the equation
\begin{equation}
\label{eqn:diamsector}
\diam S \leq (R-r)+R\theta.
\end{equation}
Furthermore, let $z$ and $w$ be complex numbers. Suppose that the angle from zero between $z$ and $w$ is~$\theta$, $0\leq\theta\leq\pi$. Then
\begin{equation}
\label{eqn:angledistance}
\abs{z}\theta\leq\pi\abs{z-w}.
\end{equation}

Take $z$, $w\in\partial X_n$ with $\abs{z-w}<\delta$. Let $r_{n,j}=\abs{w}$, $r_{n,k}=\abs{z}$; without loss of generality $j\leq k$. Then $r_{n,k}-r_{n,j}<\delta$. We wish to show that $z$ and $w$ may be connected by a continuum of diameter at most $\varepsilon$.

Either $z$ and $w$ both lie on the same side of the real axis, or they do not.

\begin{figure}
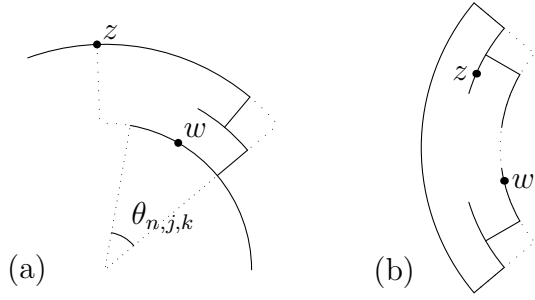

\mypic{3}
\caption{The points $z$ and $w$ in \lemref{unifcctd} and the subsets of $\partial\Omega$ connecting them. In the diagram at left, $z$ and $w$ lie on the same side of the real axis; in the diagram on the right, they lie on opposite sides of the real axis.}
\label{fig:unifcctd}
\end{figure}

If they do, then they are connected by a continuum lying in $\C\setminus\Omega_n$ that lies on the boundary of a sector of an annulus of inner radius $r_{n,j}$, outer radius $r_{n,k}$, and that subtends an angle at most $\theta_{n,j,k}+\abs{\arg z-\arg w}$. (See \figref{unifcctd}a.) By \bartoneqref{equation}{eqn:diamsector}, the diameter of this continuum is at most
\begin{align*}
(r_{n,k}-r_{n,j}) + r_{n,k}(\abs{\arg z-\arg w} + \theta_{n,j,k})
&\leq \abs{z}-\abs{w}
+\pi\abs{z-w}
+ M\theta_{n,j,k}
\\&\leq (1+\pi)\abs{z-w}+M\theta_{n,j,k}.
\end{align*}
Since $r_{n,k}-r_{n,j}<\delta\leq\delta_1(\varepsilon/MC)$, we know by Condition~\eqref{item:unifcctd:a} that $\theta_{n,j,k}<\varepsilon/MC$. Also, $\abs{z-w}<\delta<\varepsilon/C$. By choosing $C \geq 2+\pi$, we see that $z$, $w$ are connected by a continuum in $\C\setminus\Omega_n$ of diameter at most~$\varepsilon$.

If $z$ and $w$ lie on different sides of the real axis, then since $\delta\leq\mu$, $z$ and $w$ must lie on the same side of the imaginary axis. If they lie on the right-hand side, then they can be connected by a path contained in the union of their boundary arcs in $\partial X_n$ with the positive real axis. If the domain of the $\arg$ function is taken to be $(-\pi,\pi]$, then this path has length
\[
\abs{z}-\abs{w}+\abs{z} \abs{\arg z}+\abs{w} \abs{\arg w}
\leq\abs{z-w}+r_{n,k} \abs{\arg z}+r_{n,k} \abs{\arg w}
=\abs{z-w}+r_{n,k} \abs{\arg z-\arg w}
.\]
By \bartoneqref{equation}{eqn:angledistance},
\[\abs{z-w}+r_{n,k}\abs{\arg z-\arg w}\leq(1+\pi)\abs{z-w}<(1+\pi)\delta\leq (1+\pi)\varepsilon/C.\]

Otherwise, $z$ and $w$ lie on the left-hand side of the imaginary axis. We will connect them by a continuum on the left side of the circle.

Recall that $M=r_{n,n}$, and so $\theta_{n,j,n}$ measures the inset angle of the deepest gate between the arc $A_{n,j}$ and the outer boundary circle. Let
\[\psi=\min(\abs{\arg z},\abs{\arg w},\psi_{n,j}-\theta_{n,j,n}).\]
Consider the annular sector, symmetric about the real axis and containing part of the negative real axis, with inner radius $r_{n,j}$, outer radius $M=r_{n,n}$, and lying between the angles $\pm\psi$. (See \figref{unifcctd}b.)  This sector subtends an angle of $2(\pi-\psi)$ at zero. Now, $z$ and $w$ can be connected by a continuum in $\partial\Omega$ lying in this sector. By \bartoneqref{equation}{eqn:diamsector}, this path has diameter at most
\[M-r_{n,j}+2M(\pi-\psi).\]

We must control $M-r_{n,j}$ and $\pi-\psi$.

Notice that $\psi_{n,j}\leq \abs{\arg w}$ and $\psi_{n,k}\leq\abs{\arg z}$.
The angle (from zero) between $z$ and $w$ is $(\pi-\abs{\arg z})+(\pi-\abs{\arg w})$. So by \bartoneqref{equation}{eqn:angledistance},
\[
r_{n,j}((\pi-\psi_{n,j})+(\pi-\psi_{n,k}))
\leq r_{n,j} ((\pi-\abs{\arg z})+(\pi-\abs{\arg w}))
\leq \pi\abs{z-w}
\]
and so
\[\pi-\psi\leq \frac{\pi\abs{z-w}}{r_{n,j}}+\theta_{n,j,n}
< \frac{\pi\delta}{\mu} +\theta_{n,j,n}
\leq \frac{\pi\varepsilon}{C\mu}+\theta_{n,j,n}.\]

Furthermore,
\begin{align*}
\mu(\pi-\psi_{n,j})&< \pi\delta\leq \mu\,\delta_2(\varepsilon/C),\\
\mu(\pi-\psi_{n,j})&< \pi\delta\leq \mu\,\delta_2(\delta_1(\varepsilon/MC)).
\end{align*}
By \bartoneqref{condition}{item:unifcctd:b}, we have that
$M-r_{n,j}<\varepsilon/C$; since $M=r_{n,n}\geq r_{n,j}$, we have that
$\abs{r_{n,n}-r_{n,j}}\leq \delta_1(\varepsilon/MC)$. Thus by \bartoneqref{condition}{item:unifcctd:a}, we have that $\theta_{n,j,n}<\varepsilon/MC$. 

Thus, ${\pi-\psi}< \pi\varepsilon/C\mu+\varepsilon/MC$ and $M-r_{n,j}<\varepsilon/C$. Recall we can connect $z$ and $w$ by a continuum of diameter at most
\begin{align*}
M-r_{n,j}+2M(\pi-\psi)
&\leq \frac{\varepsilon}{C} + \frac{2M\pi\varepsilon}{C\mu}+\frac{2\varepsilon}{C}.
\end{align*}
Choosing $C\geq3+2\pi M/\mu$ completes the proof.
\end{proof}

%\cite{Ahl73} \bibliographystyle{amsalpha} \bibliography{bibli}  \end{document}

%%% Bibliography copied from the generated file
%%% bartonward.bbl
%%% Most citations taken from MathSciNet
%%% Some first names, my senior thesis not taken from MathSciNet

\providecommand{\bysame}{\leavevmode\hbox to3em{\hrulefill}\thinspace}
\providecommand{\MR}{\relax\ifhmode\unskip\space\fi MR }
% \MRhref is called by the amsart/book/proc definition of \MR.
\providecommand{\MRhref}[2]{%
  \href{http://www.ams.org/mathscinet-getitem?mr=#1}{#2}
}
\providecommand{\href}[2]{#2}

\end{document}